\documentclass{amsart}

\usepackage[english]{babel}

\usepackage[letterpaper,top=2cm,bottom=2cm,left=3cm,right=3cm,marginparwidth=1.75cm]{geometry}

\usepackage{amssymb}
\usepackage{amsmath}
\usepackage{amsthm}
\usepackage{color}
\usepackage{enumerate}
\usepackage{caption}
\usepackage[labelformat=simple]{subcaption}
\usepackage{booktabs}

\usepackage{tikz}
\usetikzlibrary{graphs, graphs.standard,fit,positioning}

\newtheorem{theorem}{Theorem}
\newtheorem{lemma}[theorem]{Lemma}

\title{On the minimum vertex cover of snarks}

\address{Federal University of Ceará, Campus Quixadá, Quixadá, Ceará, Brazil}
\author{Gustavo Fernandes}
\email{gustavofer14@alu.ufc.br} 

\address{Federal University of Ceará, Campus Quixadá, Quixadá, Ceará, Brazil}
\author{At\'ilio G.~Luiz}
\email{gomes.atilio@ufc.br}

\begin{document}

\begin{abstract}
A \emph{vertex cover} of a graph $G$ is a subset of vertices $C \subseteq V(G)$ such that every edge of $G$ is incident to at least one vertex in $C$. The \emph{vertex cover number} of $G$ is the minimum cardinality of a vertex cover of $G$ and is denoted by $\tau(G)$. A \emph{snark} is a connected, bridgeless, cubic graph that has an edge chromatic number of four, meaning its edges cannot be properly colored with only three colors. In this work, we investigate the problem of determining the value of a minimum vertex cover for classes of snark graphs. Given a positive integer $k$, we firstly prove that determining whether an arbitrary snark has a vertex cover $C$ with size $|C| \leq k$ is an NP-complete problem. Secondly, we determine the vertex cover number $\tau(G)$ for several subclasses of snark graphs, such as Flower snarks, Goldberg snarks, Generalized Blanu\v{s}a snarks and Loupekine snarks. 
\end{abstract}

\maketitle

\section{Introduction}
\label{sec:intro}

All graphs in this paper are simple, finite and undirected. Given a graph $G$ with vertex set $V(G)$ and edge set $E(G)$, a \emph{vertex cover} of $G$ is a set of vertices $C \subseteq V(G)$ such that, for every edge $uv \in E(G)$, we have that $u \in C$ or $v \in C$. We say that $C$ is a \emph{minimum vertex cover} of $G$ if, for every vertex cover $C'$ of $G$, it holds that $|C|\leq |C'|$. The \emph{vertex cover number} of $G$ is the cardinality of a minimum vertex cover of $G$ and is denoted by $\tau(G)$. The \textsc{Vertex Cover Problem} is a decision problem, stated as follows: given a graph $G=(V(G),E(G))$ and a positive integer $k$, does there exist a vertex cover of $G$ with cardinality at most $k$?

The study of vertex covers became prominent among graph theorists in the mid-twentieth century. In 1931, while studying matchings in bipartite graphs, the mathematician Dénes K\"{o}nig~\cite{konig1931graphok} proved that the size of a maximum matching in a bipartite graph $G$ is equal to the size of a minimum vertex cover of $G$. Subsequently, in 1959, Tibor Gallai established that the number of vertices of a graph is equal to the sum of the sizes of a maximum independent set and a minimum vertex cover~\cite{gallai1959uber}. In the 1970s, the study of vertex cover became central to the field of computational complexity theory following the work of the computer scientist Richard Karp, who identified 21 NP-complete problems, including the \textsc{Vertex Cover Problem}~\cite{Karp1972}.

The \textsc{Vertex Cover Problem} can be solved in polynomial time for graphs with maximum degree at most 2 since the connected components of these graphs are paths or cycles and the parameter $\tau(G)$ is already determined for these graphs. Consequently, a natural line of research is to investigate the problem for graphs of maximum degree 3, in particular for classes of graphs in which all vertices have degree exactly 3, the so-called \emph{cubic graphs}. However, it is known that the \textsc{Vertex Cover Problem} is NP-complete even when restricted to cubic planar graphs~\cite{GAREY1976237}.
Nevertheless, there have been significant advances in the study of the parameter 
$\tau(G)$ for certain classes of cubic graphs. A notable example is given by the generalized Petersen graphs $P(n, k)$. In particular, Behsaz et al.~\cite{behsaz2010minimumvertexcovergeneralized} and Jin et al.~\cite{JIN2019309} characterized structural properties of vertex covers in these graphs, establishing lower and upper bounds for the parameter $\tau$ for this family, as well as determining exact values of $\tau(P(n, k))$ in specific cases.

Despite the advances already achieved in the study of the \textsc{Vertex Cover Problem}, many classes of cubic graphs remain largely unexplored. One such class is that of \emph{snarks}, which are characterized as  cubic, connected, bridgeless graphs whose edges cannot be properly colored with only three colors. An example of snark, the Petersen graph, is shown in Figure~\ref{fig:petersen1}. 
Interest in these graphs arose in part due to their connection with the Four Color Theorem, which states that any planar map can be colored with at most four colors in such a way that adjacent regions receive distinct colors~\cite{Appel1977}. During attempts to prove this theorem, snarks were identified as critical cases for the failure of certain approaches based on edge colorings. Owing to their structural complexity and extreme properties, snarks play an important role as potential counterexamples or boundary cases in various optimization and decision problems on graphs~\cite{Goedgebeur2020}.

Within the family of snarks, several important subclasses stand out, such as the flower snarks~\cite{Isaacs1975}, the Goldberg snarks~\cite{GOLDBERG1981282}, the generalized Blanuša snarks~\cite{watkins1983snarks}, and the Loupekine snarks~\cite{Isaacs1976}. These families possess symmetries and recursive constructions that make them particularly appealing for the investigation of various graph parameters. For instance, several graph coloring and domination problems have been studied for these families~\cite{GONCALVES2021334,LUIZ2024260,Burdett2024,Kratica2020,Zheng2008,Zhang2014,ADAUTO2025336}. In view of the historical importance of this class in graph theory and the growing interest in studying graph parameters on it, it is natural to investigate it in the context of the Vertex Cover Problem.

In this work, we prove that determining whether an arbitrary snark has a vertex cover $C$ with size $|C| \leq k$ is an NP-complete problem, for $k \in \mathbb{N}$. We also  determine the exact value of the vertex cover number $\tau(G)$ for the subclasses: flower snarks, Goldberg snarks, generalized Blanu\v{s}a snarks and Loupekine snarks.

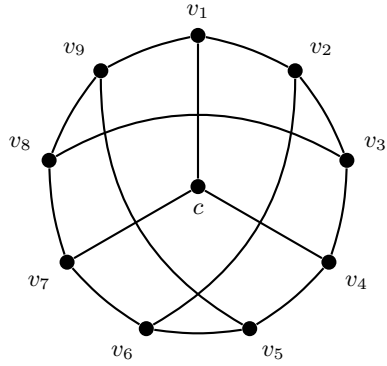
\begin{figure}
\centering

\begin{tikzpicture}[
    vertex/.style={circle, fill=black, inner sep=2pt},
    edge/.style={thick}
]

\def\R{2}

\foreach \i in {1,...,9} {
    \node[
        vertex,
        label={[font=\small] {90-40*(\i-1)}:$v_{\i}$}
    ] (v\i) at ({90-40*(\i-1)}:\R) {};
}

\node[vertex,label={[font=\small]below:$c$}] (c) at (0,0) {};

\foreach \i in {1,...,9} {
    \pgfmathtruncatemacro{\next}{mod(\i,9)+1}
    \draw[edge] (v\i) to[bend left=8] (v\next);
}

\draw[edge] (v1) -- (c);
\draw[edge] (v4) -- (c);
\draw[edge] (v7) -- (c);

\draw[edge] (v2) to[out=270,in=30] (v6);
\draw[edge] (v3) to[out=150,in=30] (v8);
\draw[edge] (v5) to[out=150,in=-90] (v9);
\end{tikzpicture}
\caption{Petersen graph $P$.}
\label{fig:petersen1}
\end{figure}

\section{Preliminaries}

In this section, we present definitions and basic results used throughout the paper. For graph-theoretic concepts not covered here, the reader is referred to~\cite{west2018introduction}.  Given a graph $G$, the goal of the vertex cover problem is to determine a minimum set of vertices that covers all edges of $G$. The parameter $\tau(G)$ corresponds to  the size of such a minimum vertex cover and satisfies several basic monotonicity and additivity properties that are frequently used in structural arguments. For example, Lemma~\ref{lemma:subgraphVC} expresses the monotonicity of $\tau(G)$ with respect to subgraphs (a minimum vertex cover of $G$ restricted to the vertices of a subgraph $H\subseteq G$ yields a vertex cover of $H$, implying $\tau(H) \leq \tau(G)$), while Lemma~\ref{lemma:VCsubgraph} provides a lower bound in terms of vertex-disjoint subgraphs (vertex-disjoint subgraphs require disjoint coverings). 

\begin{lemma}
\label{lemma:subgraphVC}
If \(H\) is a subgraph of \(G\), then \(\tau(H) \leq \tau(G)\).
\end{lemma}

\begin{lemma}
\label{lemma:VCsubgraph}
If \(G_1, G_2, \ldots, G_k\) are vertex-disjoint subgraphs of a graph \(G\), then $\tau(G) \geq \tau(G_1) + \tau(G_2) + \cdots + \tau(G_k)$.
\end{lemma}

A \emph{path} on $n\geq 2$ vertices, denoted by $P_n$, is the graph with vertex set $V(P_n) = \{v_1, v_2, \ldots, v_n\}$ and edge set $E(P_n) = \{v_iv_{i+1} \mid 1 \leq i \leq n-1\}$. A \emph{cycle} on $n \geq 3$ vertices, denoted by $C_n$, is the graph with vertex set $V(C_n) = \{v_1, v_2, \ldots, v_n\}$ and edge set $E(C_n) = \{v_iv_{i+1} \mid 1 \leq i \leq n-1\} \cup \{v_n, v_1\}$. The following standard results on the vertex cover numbers of paths and cycles will be used throughout the paper.

\begin{lemma}
\label{lemma:pathsVC}
For the path \(P_n\) with \(n \ge 2\), we have \(\tau(P_n) = \lfloor n/2 \rfloor\).
\end{lemma}

\begin{lemma}
\label{lemma:cyclesVC}
For the cycle \(C_n\) with \(n \ge 3\), we have \(\tau(C_n) = \lceil n/2 \rceil\).
\end{lemma}

A \emph{matching} in a graph $G$ is a set $M \subseteq E(G)$ of pairwise non-adjacent edges. Since no vertex can cover two edges of a matching, the size of every vertex cover of $G$ is at least the size of every matching of $G$. This is explicitly stated in the following lemma. 

\begin{lemma}[K\"{o}nig~\cite{konig1931graphok}]
\label{lemma:minMaxCover}
Let $G$ be a graph, $C$ a vertex cover of $G$, and $M$ a matching in $G$. Then $|M| \leq |C|$.
\end{lemma}

An \emph{independent set} in \(G\) is a set \(I \subseteq V(G)\) of pairwise non-adjacent vertices. A maximum independent set is one of largest cardinality; its size is denoted by \(\alpha(G)\). Independent sets and vertex covers are related concepts, as stated in the next lemma.

\begin{lemma}[Gallai~\cite{gallai1959uber}]
\label{lemma:gallai1959}
In a graph $G$, $S\subseteq V(G)$ is an independent set if and only if $V(G)\backslash S$ is a vertex cover, and hence $\alpha(G)+\tau(G) = |V(G)|$.
\end{lemma}

An \emph{edge $k$-coloring}  of $G$ is an assignment $f$ of colors from the set $\{1,2,\ldots,k\}$ to the edges of $G$. We say that the edge $k$-coloring $f$ is \emph{proper} if any two adjacent edges receive distinct colors under $f$. The \emph{chomatic index} of $G$ is the minimum $k$ for which $G$ admits a proper edge $k$-coloring,  and it is denoted by $\chi'(G)$.

\section{Hardness of the Vertex Cover Problem on Snarks}

In this section, we prove that the Vertex Cover Problem (VCP) remains NP-complete even when restricted to snarks. More specifically, we prove the following theorem.

\begin{theorem}
\label{thm:npCompleteSnarks}
The following problem is NP-Complete: for a given bridgeless cubic simple graph $G$ of girth greater than 3 with $\chi'(G)=4$ and a positive integer $k$, is there a vertex cover of $G$ with size $k$?
\end{theorem}

Since snarks are bridgeless cubic simple graphs with $\chi'(G)=4$, Theorem~\ref{thm:npCompleteSnarks} implies that the Vertex Cover Problem remains NP-complete when restricted to the class of snarks. In order to prove Theorem~\ref{thm:npCompleteSnarks}, we provide a polynomial time reduction from the following known NP-complete problem stated in Theorem~\ref{thm:npCompleteUehara}.

\begin{theorem}[Uehara~\cite{uehara1996np}]
\label{thm:npCompleteUehara}
The following problem is NP-complete: for a given 3-connected cubic planar simple graph $G$ of girth greater than 3 and a positive integer $k$, is there a vertex cover of $G$ with size $k$?\qed 
\end{theorem}

Next, we present our polynomial time reduction. 

\medskip 

\noindent \textbf{Construction of graph $G_{\mathcal{H}}$ from $G$.} Let $G$ be an arbitrary 3-connected cubic planar simple graph of girth greater than 3. 
Let the \emph{gadget} $\mathcal{H}$ be the graph defined in Figure~\ref{fig:gadget}. The degree-2 vertices of $\mathcal{H}$ are called \emph{border vertices}.
A new graph $G_{\mathcal{H}}$ is built as follows: $G_{\mathcal{H}}$ contains a disjoint copy $\mathcal{H}_w$ of $\mathcal{H}$, for each vertex $w$ of $G$.  Two copies $\mathcal{H}_u$ and $\mathcal{H}_v$ of $\mathcal{H}$ are joined by an edge whenever the
corresponding vertices $u$ and $v$ are adjacent in $G$, so that there is a one-to-one correspondence between the set of edges of $G$ and the set of edges of $G_{\mathcal{H}}$ that connect two copies of $\mathcal{H}$ (the edge that connects $\mathcal{H}_u$ to  $\mathcal{H}_v$ has as endpoints a border vertex of $\mathcal{H}_u$ and a border vertex of $\mathcal{H}_v$). We call the edges connecting copies of $\mathcal{H}$
\emph{connecting edges} of $G_{\mathcal{H}}$. The construction of $G_{\mathcal{H}}$ can clearly be done in
polynomial time in the order of $G$. Moreover, the size of $G_{\mathcal{H}}$ is also polynomial in the size of $G$. 

\medskip 

Note that, by the construction, the graph $G_{\mathcal{H}}$ is cubic. The insertion of gadgets does not introduce bridges in the constructed  graph $G_{\mathcal{H}}$ since $\mathcal{H}$ itself has no bridges. Moreover, $\mathcal{H}$ has girth 5 and the initial graph $G$ has girth greater than 3. Therefore, the construction graph $G_{\mathcal{H}}$ is a bridgeless cubic simple graph of girth greater than 3. 

In Theorem~\ref{thm:npcompleteVC}, we show that, if one can determine the vertex cover number of $G$, then it can immediately obtain the vertex cover number of $G_{\mathcal{H}}$, and vice-versa. This shows that the problem is NP-hard. However, in order to prove Theorem~\ref{thm:npcompleteVC}, we need the next two auxiliary lemmas, that present useful properties of the gadget $\mathcal{H}$.

\begin{figure}
\centering

\begin{tikzpicture}[
    vertex/.style={circle, fill=black, inner sep=2pt},
    edge/.style={thick}
]

\def\R{1.3}

\def\Rbig{2.5}

\foreach \i in {1,...,6} {
    \node[
        vertex,
        label={[font=\small] {90-60*(\i-1)}:$u_{\i}$}
    ] (u\i) at ({90-60*(\i-1)}:\Rbig) {};
}

\foreach \i in {1,...,6} {
    \pgfmathtruncatemacro{\next}{mod(\i,6)+1}
    \draw[edge] (u\i) -- (u\next);
}


\foreach \i in {1,...,9} {

    \ifnum\i=1
        \node[vertex,
        label={[font=\small] 140:$v_{\i}$}]
        (v\i) at ({90-40*(\i-1)}:\R) {};
    \else\ifnum\i=4
        \node[vertex,
        label={[font=\small] 7:$v_{\i}$}]
        (v\i) at ({90-40*(\i-1)}:\R) {};
    \else\ifnum\i=7
        \node[vertex,
        label={[font=\small] 160:$v_{\i}$}]
        (v\i) at ({90-40*(\i-1)}:\R) {};
    \else
        \node[vertex,
        label={[font=\small] {90-40*(\i-1)}:$v_{\i}$}]
        (v\i) at ({90-40*(\i-1)}:\R) {};
    \fi\fi\fi
}

\foreach \i in {1,...,9} {
    \pgfmathtruncatemacro{\next}{mod(\i,9)+1}
    \draw[edge] (v\i) to[bend left=8] (v\next);
}

\draw[edge] (v2) to[out=270,in=30] (v6);
\draw[edge] (v3) to[out=150,in=30] (v8);
\draw[edge] (v5) to[out=150,in=-90] (v9);

\draw[edge] (v1) -- (u1);
\draw[edge] (v4) -- (u3);
\draw[edge] (v7) -- (u5);

\end{tikzpicture}

\caption{Gadget $\mathcal{H}$. The vertices of degree two ($u_2,u_4,u_6$) are the border vertices.}
\label{fig:gadget}
\end{figure}
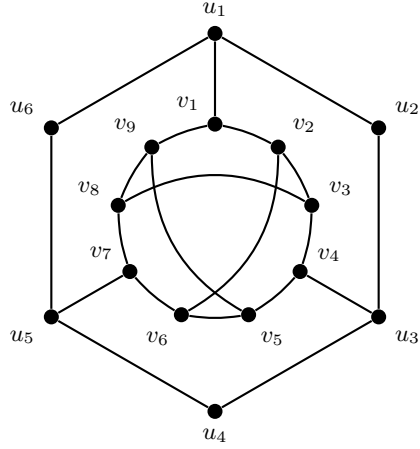

\begin{lemma}
\label{lemma:VCgadget}
The gadget $\mathcal{H}$ has $\tau(\mathcal{H})=8$.
\end{lemma}

\begin{proof}
The gadget $\mathcal{H}$ has two vertex-disjoint induced cycles: a 6-cycle induced by $\{u_1,u_2,\ldots,u_6\}$ and a 9-cycle induced by $\{v_1,v_2,\ldots,v_9\}$. Thus, by Lemmas~\ref{lemma:VCsubgraph} and~\ref{lemma:cyclesVC}, we obtain that $\tau(\mathcal{H}) \geq \tau(C_6)+\tau(C_9) = 3 + 5 = 8$. To conclude, a vertex cover of $\mathcal{H}$ with cardinality 8 is shown in Figure~\ref{fig:vcGadget1}.
\end{proof}

\begin{lemma}
\label{lemma:VCgadget2}
Let $\mathcal{H}$ be the gadget shown in Figure~\ref{fig:gadget}. If $S$ is a minimum vertex cover of $\mathcal{H}$, then none of the border vertices $u_2$, $u_4$, and $u_6$ belong to the set $S$.
\end{lemma}

\begin{proof}
Let $S$ be a minimum vertex cover for the graph $\mathcal{H}$. By Lemma~\ref{lemma:VCgadget}, we have that $|S|=8$. 
Suppose, for a contradiction, that $u_2 \in S$. We know that the 6-cycle induced by the vertices in $\{u_1,\ldots,u_6\}$ must be covered by at least 3 vertices and the 9-cycle induced by the vertices in $\{v_1,\ldots,v_9\}$ must be covered by at least 5 vertices. These facts imply that exactly two other vertices of the 6-cycle must belong to $S$, which implies that $u_4,u_6 \in S$ in order to cover all the edges of the outer 6-cycle. Note that the edges $u_1v_1$, $u_3v_4$ and $u_5v_7$ are not covered by none of the vertices $u_2,u_4,u_6$ and, then, must be covered by vertices of the inner 9-cycle, which implies that $v_1,v_4,v_7 \in S$. At this point, there exists in $\mathcal{H}$ a set of 6 edges inducing a 6-cycle, none of which are yet covered; the edges are $\{v_8v_9,v_9v_5,v_5v_6,v_6v_2,v_2v_3,v_3v_8\}$. By Lemma~\ref{lemma:cyclesVC}, this 6-cycle requires at least 3 more vertices to be covered, thus contradicting the fact that $|S| = 8$. Therefore, $u_2 \not\in S$. By the symmetry of the gadget $\mathcal{H}$, the same conclusion applies to the other border vertices $u_4$ and $u_6$.
\end{proof}

Now, we are ready to prove Theorem~\ref{thm:npcompleteVC}.

\begin{theorem}
\label{thm:npcompleteVC}
Let $G$ be a 3-connected cubic simple planar graph of girth greater than 3. Let $G_{\mathcal{H}}$ be the bridgeless cubic simple graph of girth greater than 3 obtained from $G$ by replacing every vertex $v \in V(G)$ for a copy $\mathcal{H}_v$ of  the gadget $\mathcal{H}$. Then, $\tau(G_{\mathcal{H}}) = 8|V(G)|+\tau(G)$.
\end{theorem}

\begin{proof}
Let $G$ and $G_{\mathcal{H}}$ be as declared in the statement. In order to show that $\tau(G_{\mathcal{H}}) \leq  8|V(G)|+\tau(G)$, we construct a vertex cover $S$ for $G_{\mathcal{H}}$ such that $|S| \leq 8|V(G)|+\tau(G)$. Let $C \subseteq V(G)$ be a minimum vertex cover of $G$, i.e., $|C| = \tau(G)$. Recall that each vertex $w$ of $G$ is replaced by a copy $\mathcal{H}_w$ of the gadget $\mathcal{H}$ so as to form the graph $G_{\mathcal{H}}$. If $w \not\in C$, then we choose the 8 vertices $\{u_1,u_3,u_5,v_2,v_3,v_5,v_7,v_9\} \subset V(\mathcal{H}_w)$ to compose the vertex cover $S$ (see Figure~\ref{fig:vcGadget1}). On the other hand, if $w \in C$, then we choose the 9 vertices $\{u_2,u_4,u_6,v_1,v_3,v_4,v_6,v_7,v_9\} \subset V(\mathcal{H}_w)$ to compose the vertex cover $S$ (see Figure~\ref{fig:vcGadget2}). From Figures~\ref{fig:vcGadget1} and~\ref{fig:vcGadget2}, it can easily be checked that these chosen vertices cover each of the copies $\mathcal{H}_w$. In order to prove that the set $S$ just defined is a vertex cover of $G_{\mathcal{H}}$, it remains to show that an arbitrary edge connecting two distinct gagdets  $\mathcal{H}_x$ and $\mathcal{H}_y$ is covered by $S$. 
Let $xy \in E(G)$ and let $\mathcal{H}_x$ and $\mathcal{H}_y$ be the gadgets of $G_{\mathcal{H}}$ corresponding to the vertices $x,y \in V(G)$. 
From the definition of the construction, the edge connecting  $\mathcal{H}_x$ and $\mathcal{H}_y$ links a border vertex $u_i \in V(\mathcal{H}_x)$ to a border vertex $u_j \in V(\mathcal{H}_y)$, for $i,j \in\{2,4,6\}$. 
Moreover, from the construction of $S$, at least one of the gadgets $\mathcal{H}_x$ and $\mathcal{H}_y$ came from a vertex of $G$ belonging to the vertex cover $C$. This implies that at least one of these gadgets is covered by the vertex cover illustrated in Figure~\ref{fig:vcGadget2} and, hence, at least one of the vertices $u_i$ and $u_j$ belongs to the vertex cover $S$. Therefore, the edge connecting $\mathcal{H}_x$ and $\mathcal{H}_y$ is covered by $S$. The order of the vertex cover $S$ is given by $|S| = 9|C| + 8(|V(G)|-|C|) = 8|V(G)|+|C|$. Therefore, $\tau(G_{\mathcal{H}}) \leq |S| =  8|V(G)|+\tau(G)$.

Next, we show that $\tau(G_{\mathcal{H}}) \geq   8|V(G)|+\tau(G)$. Let $S$ be a minimum vertex cover of $G_{\mathcal{H}}$. 
Suppose, for the sake of contradiction, that $|S| = 8|V(G)|+\tau(G)-k$, for some $k \geq 1$. 
By the construction, $G_{\mathcal{H}}$ contains $|V(G)|$ vertex-disjoint copies of the gagdet $\mathcal{H}$ and each one of these copies needs at least 8 vertices to be covered (see Lemma~\ref{lemma:VCgadget}).   
However, by Lemma~\ref{lemma:VCgadget2}, in each copy of the gadget $\mathcal{H}$ that is covered by exactly 8 vertices of $S$, the border vertices of the copy do not belong to the vertex cover $S$. So, it is not possible to have all the gadgets covered by exactly 8 vertices, since the edges connecting two distinct gadgets would not be covered. Therefore, the additional $\tau(G)-k$ vertices of $S$ are responsible to cover the edges that link the vertex-disjoint gadgets of $G_{\mathcal{H}}$ (these vertices are necessarily border vertices). However, from these $\tau(G)-k$ vertices, we obtain a vertex cover $W$ of $G$ with size $\tau(G)-k$ (each vertex $x\in V(G)$ whose corresponding gadget $\mathcal{H}_x$ has a border vertex in the set $S$ belongs to $W$), which is a contradiction. Therefore, $\tau(G_{\mathcal{H}}) \geq   8|V(G)|+\tau(G)$.

From the lower and upper bounds, we conclude that $\tau(G_{\mathcal{H}}) = 8|V(G)|+\tau(G)$.
\end{proof}

In order to conclude the proof of Theorem~\ref{thm:npCompleteSnarks}, it remains to show that the constructed graph $G_{\mathcal{H}}$ is a snark. In order to do this, we need an auxiliary definition.

Let $G$ and $H$ be two vertex-disjoint cubic graphs. Let $u \in V(G)$ and $v \in V(H)$ be arbitrary vertices, and let $N_{G}(u) = \{w_1,w_2,w_3\}$ and $N_{H}(v) = \{z_1,z_2,z_3\}$. 
As defined by Holton et al.~\cite{holton1993petersen}, a  \emph{star product}, denoted by $G \star H$, of the two cubic graphs $G$ and $H$ is a cubic graph obtained by deleting the vertices $u$ and $v$ from $G$ and $H$, respectively, and by adding the edges $w_1z_1, w_2z_2, w_3z_3$. A scheme of the star product operation is shown in Figure~\ref{fig:starProduct}.

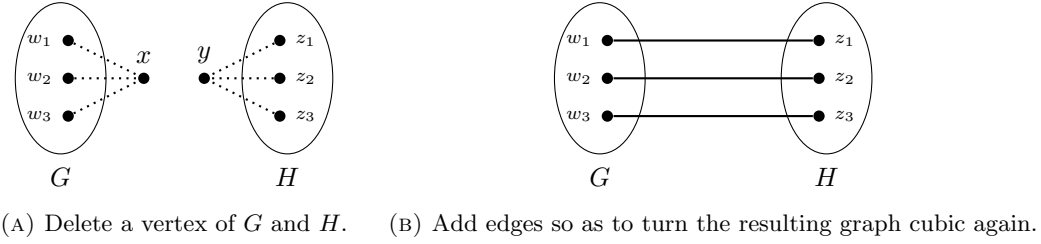
\begin{figure}
\centering
\begin{subfigure}[b]{0.3\textwidth}
\centering
\begin{tikzpicture}[scale=.5]

\tikzset{
    vertex/.style={circle, fill=black, inner sep=1.5pt},
    edge/.style={dotted, thick}
}

\draw (-3,0) ellipse (1.2 and 2);

\node at (-3,-2.6) {$G$};

\node[vertex,label=left:{\scriptsize $w_1$}] (g1) at (-2.8,1) {};
\node[vertex,label=left:{\scriptsize $w_2$}] (g2) at (-2.8,0) {};
\node[vertex,label=left:{\scriptsize $w_3$}] (g3) at (-2.8,-1) {};

\node[vertex,label={$x$}] (x) at (-0.8,0) {};

\draw[edge] (g1) -- (x);
\draw[edge] (g2) -- (x);
\draw[edge] (g3) -- (x);

\draw (3,0) ellipse (1.2 and 2);

\node at (3,-2.6) {$H$};

\node[vertex,label=right:{\scriptsize $z_1$}] (h1) at (2.8,1) {};
\node[vertex,label=right:{\scriptsize $z_2$}] (h2) at (2.8,0) {};
\node[vertex,label=right:{\scriptsize $z_3$}] (h3) at (2.8,-1) {};

\node[vertex,label={$y$}] (y) at (0.8,0) {};

\draw[edge] (h1) -- (y);
\draw[edge] (h2) -- (y);
\draw[edge] (h3) -- (y);

\end{tikzpicture}

\caption{Delete a vertex of $G$ and $H$.}
\label{fig:starProduct1}
\end{subfigure}
\begin{subfigure}[b]{0.6\textwidth}
\centering
\begin{tikzpicture}[scale=.5]

\tikzset{
    vertex/.style={circle, fill=black, inner sep=1.5pt},
    edge/.style={thick}
}

\draw (-3,0) ellipse (1.2 and 2);

\node at (-3,-2.6) {$G$};

\node[vertex,label=left:{\scriptsize $w_1$}] (g1) at (-2.8,1) {};
\node[vertex,label=left:{\scriptsize $w_2$}] (g2) at (-2.8,0) {};
\node[vertex,label=left:{\scriptsize $w_3$}] (g3) at (-2.8,-1) {};

\draw (3,0) ellipse (1.2 and 2);

\node at (3,-2.6) {$H$};

\node[vertex,label=right:{\scriptsize $z_1$}] (h1) at (2.8,1) {};
\node[vertex,label=right:{\scriptsize $z_2$}] (h2) at (2.8,0) {};
\node[vertex,label=right:{\scriptsize $z_3$}] (h3) at (2.8,-1) {};

\draw[edge] (h1) to (g1);
\draw[edge] (h2) to (g2);
\draw[edge] (h3) to (g3);

\end{tikzpicture}

\caption{Add edges so as to turn the resulting graph cubic again.}
\label{fig:starProduct2}
\end{subfigure}
\caption{Schematization of the star product operation $G\star H$ between two cubic graphs $G$ and $H$.}
\label{fig:starProduct}
\end{figure}

The next lemma is a useful result about star products.

\begin{lemma}[Issacs~\cite{Isaacs1975}, Holton et al.~\cite{holton1993petersen}]
\label{lemma:starProduct1}
If a cubic graph $F$ is obtained by the star product of two bridgeless cubic simple graphs $G$ and $H$, such that at least one of $G$ or $H$ is a snark, then $F$ itself is also a snark.\qed 
\end{lemma}

Let $\mathcal{H}$ be the gadget shown in Figure~\ref{fig:gadget}. Note that the subgraph of $\mathcal{H}$ induced by the vertices from the set $W = \{v_1,v_2,\ldots,v_9\}$ is the Petersen graph $P$, shown in Figure~\ref{fig:petersen1}, with the vertex $c$ deleted. That is, $P-c = G[W]$. 

Let $G_{\mathcal{H}}^-$ be the cubic graph obtained from $G_{\mathcal{H}}$ by replacing each subgraph $G[W]$ by a new vertex in all copies of $\mathcal{H}$. The graph $G_{\mathcal{H}}$ is obtained
by $|V(G)|$ applications of the star product of the Petersen graph $P$ and $G_{\mathcal{H}}^-$. By Lemma~\ref{lemma:starProduct1}, since the Petersen graph is a snark, the graph $G_{\mathcal{H}}$ is a snark. This concludes the proof of Theorem~\ref{thm:npCompleteSnarks}.

\begin{figure}
\centering
\begin{subfigure}[b]{0.4\textwidth}
\centering
\begin{tikzpicture}[
    vertexB/.style={circle, draw=black, fill=black, inner sep=2pt},
    vertexW/.style={circle, draw=black, fill=white, inner sep=2pt},
    edge/.style={thick}
]

\def\R{1.3}

\def\Rbig{2.5}

\foreach \i in {1,...,6} {

    \ifodd\i
        \def\ustyle{vertexB}
    \else
        \def\ustyle{vertexW}
    \fi

    \node[
        \ustyle,
        label={[font=\small] {90-60*(\i-1)}:$u_{\i}$}
    ] (u\i) at ({90-60*(\i-1)}:\Rbig) {};
}

\foreach \i in {1,...,6} {
    \pgfmathtruncatemacro{\next}{mod(\i,6)+1}
    \draw[edge] (u\i) -- (u\next);
}


\foreach \i in {1,...,9} {

    \ifnum\i=2 \def\vstyle{vertexB}
    \else\ifnum\i=3 \def\vstyle{vertexB}
    \else\ifnum\i=5 \def\vstyle{vertexB}
    \else\ifnum\i=7 \def\vstyle{vertexB}
    \else\ifnum\i=9 \def\vstyle{vertexB}
    \else \def\vstyle{vertexW}
    \fi\fi\fi\fi\fi

    \ifnum\i=1
        \node[\vstyle,
        label={[font=\small] 140:$v_{\i}$}]
        (v\i) at ({90-40*(\i-1)}:\R) {};
    \else\ifnum\i=4
        \node[\vstyle,
        label={[font=\small] 7:$v_{\i}$}]
        (v\i) at ({90-40*(\i-1)}:\R) {};
    \else\ifnum\i=7
        \node[\vstyle,
        label={[font=\small] 160:$v_{\i}$}]
        (v\i) at ({90-40*(\i-1)}:\R) {};
    \else
        \node[\vstyle,
        label={[font=\small] {90-40*(\i-1)}:$v_{\i}$}]
        (v\i) at ({90-40*(\i-1)}:\R) {};
    \fi\fi\fi
}

\foreach \i in {1,...,9} {
    \pgfmathtruncatemacro{\next}{mod(\i,9)+1}
    \draw[edge] (v\i) to[bend left=8] (v\next);
}

\draw[edge] (v2) to[out=270,in=30] (v6);
\draw[edge] (v3) to[out=150,in=30] (v8);
\draw[edge] (v5) to[out=150,in=-90] (v9);

\draw[edge] (v1) -- (u1);
\draw[edge] (v4) -- (u3);
\draw[edge] (v7) -- (u5);

\end{tikzpicture}

\caption{Vertex cover of size 8.}
\label{fig:vcGadget1}
\end{subfigure}
\begin{subfigure}[b]{0.4\textwidth}
\centering
\begin{tikzpicture}[
    vertexB/.style={circle, draw=black, fill=black, inner sep=2pt},
    vertexW/.style={circle, draw=black, fill=white, inner sep=2pt},
    edge/.style={thick}
]

\def\R{1.3}

\def\Rbig{2.5}

\foreach \i in {1,...,6} {

    \ifnum\i=2 \def\ustyle{vertexB}
    \else\ifnum\i=4 \def\ustyle{vertexB}
    \else\ifnum\i=6 \def\ustyle{vertexB}
    \else \def\ustyle{vertexW}
    \fi\fi\fi

    \node[
        \ustyle,
        label={[font=\small] {90-60*(\i-1)}:$u_{\i}$}
    ] (u\i) at ({90-60*(\i-1)}:\Rbig) {};
}

\foreach \i in {1,...,6} {
    \pgfmathtruncatemacro{\next}{mod(\i,6)+1}
    \draw[edge] (u\i) -- (u\next);
}


\foreach \i in {1,...,9} {

    \ifnum\i=1 \def\vstyle{vertexB}
    \else\ifnum\i=3 \def\vstyle{vertexB}
    \else\ifnum\i=4 \def\vstyle{vertexB}
    \else\ifnum\i=6 \def\vstyle{vertexB}
    \else\ifnum\i=7 \def\vstyle{vertexB}
    \else\ifnum\i=9 \def\vstyle{vertexB}
    \else \def\vstyle{vertexW}
    \fi\fi\fi\fi\fi\fi

    \ifnum\i=1
        \node[\vstyle,
        label={[font=\small] 140:$v_{\i}$}]
        (v\i) at ({90-40*(\i-1)}:\R) {};
    \else\ifnum\i=4
        \node[\vstyle,
        label={[font=\small] 7:$v_{\i}$}]
        (v\i) at ({90-40*(\i-1)}:\R) {};
    \else\ifnum\i=7
        \node[\vstyle,
        label={[font=\small] 160:$v_{\i}$}]
        (v\i) at ({90-40*(\i-1)}:\R) {};
    \else
        \node[\vstyle,
        label={[font=\small] {90-40*(\i-1)}:$v_{\i}$}]
        (v\i) at ({90-40*(\i-1)}:\R) {};
    \fi\fi\fi
}

\foreach \i in {1,...,9} {
    \pgfmathtruncatemacro{\next}{mod(\i,9)+1}
    \draw[edge] (v\i) to[bend left=8] (v\next);
}

\draw[edge] (v2) to[out=270,in=30] (v6);
\draw[edge] (v3) to[out=150,in=30] (v8);
\draw[edge] (v5) to[out=150,in=-90] (v9);

\draw[edge] (v1) -- (u1);
\draw[edge] (v4) -- (u3);
\draw[edge] (v7) -- (u5);

\end{tikzpicture}

\caption{Vertex cover of size 9.}
\label{fig:vcGadget2}
\end{subfigure}
    
\caption{Two distinct vertex covers of the same graph. The vertices of each  vertex cover are shown in black.}
\label{fig:vcgadgets}
\end{figure}
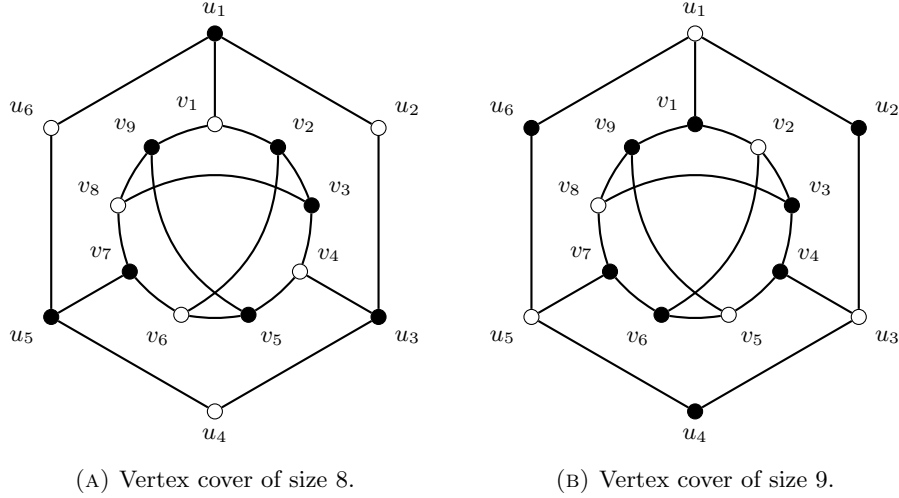

\section{Flower Snarks}

Flower snarks constitute an infinite family of snarks introduced by Rufus Isaacs in 1975~\cite{Isaacs1975}. For any odd integer $n \geq 3$, the \emph{flower snark} $F_n$ is the graph constructed from the disjoint union of $n$ subgraphs $B_1,B_2,\ldots,B_n$ called building blocks. Each \emph{building block} $B_j$, for $j \in \{1, 2, \ldots, n\}$, has vertex set $V(B_j) = \{x_j, y_j, v_j, u_j\}$ and edge set $E(B_j) = \{x_jv_j, y_jv_j, u_jv_j\}$, as shown in Figure~\ref{fig:bloco_basico}. Any two consecutive blocks $B_k$ and $B_{k+1}$ are joined by \emph{link-edges} from the set $E_{k,k+1} = \{u_ku_{k+1}, x_kx_{k+1}, y_ky_{k+1}\}$, as illustrated in Figure~\ref{fig:grafo_ligacao}. We say that $B_j$ is \emph{odd} whenever j is odd, and \emph{even} otherwise. For odd $n\geq 3$, the vertex set of the flower snark $F_n$ is given by $V(F_n) = \bigcup_{i=1}^{n} V(B_i)$ and its edge set is given by $E(F_n) = \left(\bigcup_{i=1}^{n} E(B_i)\right) \cup \left(\bigcup_{k=1}^{n-1} E_{k,k+1}\right) \cup \{y_1x_n, x_1y_n, u_1u_n\}$. The flower snark $F_3$ is exhibited in Figure~\ref{fig:snark_f3}. By the definition of $F_n$, we have that $|V(F_n)|=4n$ and $|E(F_n)|=6n$.

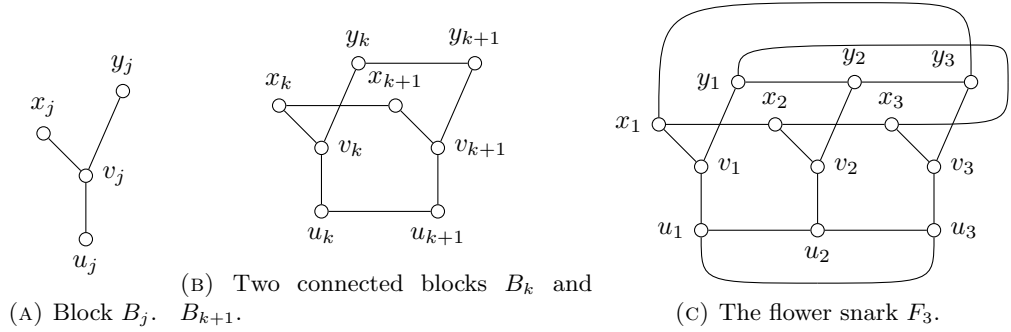
\begin{figure}
\centering
\begin{subfigure}[b]{0.15\textwidth}
\centering
\begin{tikzpicture}[cp/.style={circle,fill=white,draw,minimum size=.5em,inner sep=0pt},scale=.7]
\node[cp,label=south:$u_j$] (ui) at (0,0) {};
\node[cp,label=east:$v_j$] (vi) at (0,1.2) {};
\node[cp,label=north:$x_j$] (xi) at (-.8,2) {};
\node[cp,label=north:$y_j$] (yi) at (.7,2.8) {};
\draw (ui) -- (vi) -- (yi)  (vi) -- (xi);
\end{tikzpicture}
\caption{Block \(B_j\).}
\label{fig:bloco_basico}
\end{subfigure}
\begin{subfigure}[b]{0.35\textwidth}
\centering
\begin{tikzpicture}[cp/.style={circle,fill=white,draw,minimum size=.5em,inner sep=0pt},scale=.7]
            \node[cp,label=south:$u_k$] (ui) at (0,0) {};
            \node[cp,label=east:$v_k$] (vi) at (0,1.2) {};
            \node[cp,label=north:$x_k$] (xi) at (-.8,2) {};
            \node[cp,label=north:$y_k$] (yi) at (.7,2.8) {};
            
            \node[cp,label=south:$u_{k+1}$] (ui1) at (2.2,0) {};
            \node[cp,label=east:$v_{k+1}$] (vi1) at (2.2,1.2) {};
            \node[cp,label=north:$x_{k+1}$] (xi1) at (1.4,2) {};
            \node[cp,label=north:$y_{k+1}$] (yi1) at (2.9,2.8) {};
            
            \draw (ui) -- (vi) -- (yi)  (vi) -- (xi);
            \draw (ui1) -- (vi1) -- (yi1)  (vi1) -- (xi1);
            \draw (ui) -- (ui1)   (yi)--(yi1)  (xi)--(xi1);
        \end{tikzpicture}
\caption{Two connected blocks $B_k$ and $B_{k+1}$.}
\label{fig:grafo_ligacao}
\end{subfigure}
\begin{subfigure}[b]{0.35\textwidth}
\centering
\begin{tikzpicture}[cp/.style={circle,fill=white,draw,minimum size=.5em,inner sep=0pt},scale=.7]
            \node[cp,label=west:$u_1$] (ui) at (0,0) {};
            \node[cp,label=east:$v_1$] (vi) at (0,1.2) {};
            \node[cp,label=west:$x_1$] (xi) at (-.8,2) {};
            \node[cp,label=west:$y_1$] (yi) at (.7,2.8) {};
            
            \node[cp,label=south:$u_2$] (ui1) at (2.2,0) {};
            \node[cp,label=east:$v_2$] (vi1) at (2.2,1.2) {};
            \node[cp,label=north:$x_2$] (xi1) at (1.4,2) {};
            \node[cp,label=north:$y_2$] (yi1) at (2.9,2.8) {};
            
            \node[cp,label=east:$u_3$] (ui2) at (4.4,0) {};
            \node[cp,label=east:$v_3$] (vi2) at (4.4,1.2) {};
            \node[cp,label=north:$x_3$] (xi2) at (3.6,2) {};
            \node[cp,label=north west:$y_3$] (yi2) at (5.1,2.8) {};
            
            \draw (ui) -- node[right] {} (vi);
            \draw (vi) -- node[right, pos=.2] {} (yi);
            \draw (vi) -- node[left] {} (xi);
            \draw (ui1) -- node[right] {} (vi1); 
            \draw (vi1) -- node[right, pos=.2] {} (yi1);
            \draw (vi1) -- node[left] {} (xi1);
            \draw (ui2) -- node[right] {} (vi2); 
            \draw (vi2) -- node[right, pos=.2] {} (yi2);
            \draw (vi2) -- node[left] {} (xi2);
            \draw (ui) -- node[below] {} (ui1);
            \draw (yi)-- node[above] {} (yi1);
            \draw (xi)-- node[above, pos=.8] {} (xi1);
            \draw (ui2) -- node[below] {} (ui1);
            \draw (yi2) -- node[above] {} (yi1);
            \draw (xi2) -- node[above, pos=.2] {} (xi1);
            
            \draw (ui) .. controls (0,-1) ..  (2.2,-1);
            \draw (2.2,-1) .. controls (4.4,-1) .. (ui2);
            
            \draw (xi) .. controls (-.8,4.3) .. (3,4.3)
                  (3,4.3) .. controls (5.1,4.3) .. (yi2);
            \draw (yi) .. controls (.7,3.5) .. (5.1,3.5);
            \draw (5.1,3.5) .. controls (5.8,3.5) .. (5.8,2.8);
            \draw (5.8,2.8) .. controls (5.8,2) .. (xi2);
        \end{tikzpicture}
\caption{The flower snark \(F_3\).}
\label{fig:snark_f3}
\end{subfigure}
\caption{The process of constructing the flower snark $F_3$ from its building blocks.}
\label{fig:elementos_snark}
\end{figure}

Theorem~\ref{the:flower-snark-proof} determines the vertex cover number of flower snarks, relying on classical results in graph theory relating minimum vertex covers to maximum independent sets and also in the following characterization of bipartite graphs.

\begin{lemma}[K\"{o}nig~\cite{konig2001theorie}]
\label{lemma:bipartiteSse}
A graph is bipartite if and only if it contains no odd cycles.
\end{lemma}

\begin{theorem}
\label{the:flower-snark-proof}
Let $n\geq 3$ be an odd integer. If $F_n$ is a flower snark, then, $\tau(F_n) = 2n + 1$.
\end{theorem}

\begin{proof}
Let $F_n$ be a flower snark with odd $n \geq 3$. Since $|V(F_n)| = 4n$, Lemma~\ref{lemma:gallai1959} implies that $\tau(F_n) = 4n - \alpha(F_n)$. Therefore, to conclude that $\tau(F_n) = 2n + 1$, it suffices to prove that $\alpha(F_n) = 2n - 1$. We begin by establishing the lower bound $\alpha(F_n) \geq 2n - 1$. In order to do this, we define an independent set $I(F_n)$ for the flower snark  $F_n$ with $2n-1$ vertices as follows: 
\[
I(F_n)=\left\{v_{2i+1}\mid 0\leq i\leq \frac{n-1}           {2}\right\}\bigcup\left\{y_{2j},x_{2j},u_{2j}\mid 1\leq j\leq \frac{n-1}{2}\right\}.
\]

As an example, Figure~\ref{fig:snark_f3_independent_set} illustrates $F_3$ with the independent set $I(F_3) = \{v_1, y_2, x_2, u_2, v_3\}$.

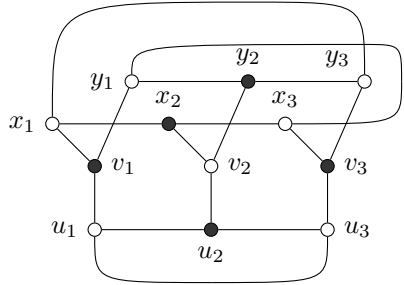
\begin{figure}
\centering
    \begin{tikzpicture}[cp/.style={circle,fill=white,draw,minimum size=.5em,inner sep=0pt},scale=0.7]
        \node[cp,label=west:$u_1$] (ui) at (0,0) {};
        \node[cp,fill=black!80,label=east:$v_1$] (vi) at (0,1.2) {};
        \node[cp,label=west:$x_1$] (xi) at (-.8,2) {};
        \node[cp,label=west:$y_1$] (yi) at (.7,2.8) {};
        
        \node[cp,fill=black!80,label=south:$u_2$] (ui1) at (2.2,0) {};
        \node[cp,label=east:$v_2$] (vi1) at (2.2,1.2) {};
        \node[cp,fill=black!80,label=north:$x_2$] (xi1) at (1.4,2) {};
        \node[cp,fill=black!80,label=north:$y_2$] (yi1) at (2.9,2.8) {};
        
        \node[cp,label=east:$u_3$] (ui2) at (4.4,0) {};
        \node[cp,fill=black!80,label=east:$v_3$] (vi2) at (4.4,1.2) {};
        \node[cp,label=north:$x_3$] (xi2) at (3.6,2) {};
        \node[cp,label=north west:$y_3$] (yi2) at (5.1,2.8) {};
        
        \draw (ui) -- node[right] {} (vi);
        \draw (vi) -- node[right, pos=.2] {} (yi);
        \draw (vi) -- node[left] {} (xi);
        \draw (ui1) -- node[right] {} (vi1); 
        \draw (vi1) -- node[right, pos=.2] {} (yi1);
        \draw (vi1) -- node[left] {} (xi1);
        \draw (ui2) -- node[right] {} (vi2); 
        \draw (vi2) -- node[right, pos=.2] {} (yi2);
        \draw (vi2) -- node[left] {} (xi2);
        \draw (ui) -- node[below] {} (ui1);
        \draw (yi)-- node[above] {} (yi1);
        \draw (xi)-- node[above, pos=.8] {} (xi1);
        \draw (ui2) -- node[below] {} (ui1);
        \draw (yi2) -- node[above] {} (yi1);
        \draw (xi2) -- node[above, pos=.2] {} (xi1);
        
        \draw (ui) .. controls (0,-1) ..  (2.2,-1);
        \draw (2.2,-1) .. controls (4.4,-1) .. (ui2);
        
        \draw (xi) .. controls (-.8,4.3) .. (3,4.3)
              (3,4.3) .. controls (5.1,4.3) .. (yi2);
        \draw (yi) .. controls (.7,3.5) .. (5.1,3.5);
        \draw (5.1,3.5) .. controls (5.8,3.5) .. (5.8,2.8);
        \draw (5.8,2.8) .. controls (5.8,2) .. (xi2);
    \end{tikzpicture}
\caption{An independent set in the flower snark \(F_3\) indicated by the vertices in black color.}
\label{fig:snark_f3_independent_set}
\end{figure}

\vspace{0.5cm}

In order to see that $|I(F_n)|=2n-1$, note that $I(F_n)$ contains one vertex from each odd block and three vertices from each even block. Since $n$ is an odd number, there are $\frac{n+1}{2}$ odd blocks and $\frac{n-1}{2}$ even blocks. Summing up these contributions, we obtain:$$|I(F_n)| = 1 \cdot \left(\frac{n+1}{2}\right) + 3 \cdot \left(\frac{n-1}{2}\right) = \frac{n+1 + 3n - 3}{2} = \frac{4n - 2}{2} = 2n - 1.$$

Now, we prove that $I(F_n)$ is an independent set. Note that each even block is connected by link-edges only to odd blocks, and vice versa. Moreover, $v_i$ is the only vertex of an odd block $B_i$ that belongs to $I(F_n)$ and this vertex is not adjacent to any vertex of an even block. Furthermore, the three vertices of each even block $B_i$ that are chosen to form $I(F_n)$, are the vertices $u_i$, $x_i$ and $y_i$, that are pairwise non-adjacent. Therefore, $I(F_n)$ is an independent set with $|I(F_n)|= 2n-1$.

Next, we establish the upper bound $\alpha(F_n) \leq 2n-1$. Suppose, for the sake of contradiction, that $F_n$ has an independent set $I(F_n)$ of size $|I(F_n)| \geq 2n$. Partition $V(F_n)$ into two sets $I(F_n)$ and $C(F_n)$ such that $C(F_n) = V(F_n) \setminus I(F_n)$. By Lemma~\ref{lemma:gallai1959}, $C(F_n)$ is a vertex cover of $F_n$ (every edge in $F_n$ must have at least one endpoint in $C(F_n)$). Since $|V(F_n)| = 4n$, then $|C(F_n)| = |V(F_n)|-|I(F_n)| \leq 2n$.  

Since $F_n$ is a cubic graph, we can calculate the sum of the degrees of the vertices in the sets $I(F_n)$ and $C(F_n)$ as follows:$$\sum_{u \in I(F_n)} d(u) = 3\cdot |I(F_n)| \geq 6n \quad \text{and} \quad \sum_{v \in C(F_n)} d(v) = 3\cdot |C(F_n)| \leq 6n.$$

Since $|E(F_n)| = 6n$ and the sum of degrees of the vertices in the independent set $I(F_n)$ is at least $6n$, then all the edges of $F_n$ have an endpoint in $I(F_n)$ and the other endpoint in the vertex cover $C(F_n)$. This implies that  $C(F_n)$ is an independent set, turning $F_n$ into a bipartite graph. This contradicts Lemma~\ref{lemma:bipartiteSse}, as $F_n$ contains an induced odd cycle of length $n$ formed by the vertices $\{u_1, u_2, \ldots, u_n\}$.

From this contradiction, we obtain that our initial assumption is false. Hence, $F_n$ cannot contain an independent set of size $2n$ or larger. Combining the previously established lower bound $\alpha(F_n) \geq 2n-1$ with this upper bound $\alpha(F_n) \leq 2n-1$, we conclude that $\alpha(F_n) = 2n-1$. This, in turn, yields $\tau(F_n) = 2n+1$, completing the proof.
\end{proof}

\section{Goldberg Snarks}

The family of Goldberg snarks was built around 1981 by Mark K.~Goldberg~\cite{GOLDBERG1981282} and consists of the graphs $G_n$ for odd $n \geq 3$. Similar to flower snarks, they are constructed from the disjoint union of subgraphs called building blocks. For each $i \in \{1,\ldots,n\}$, a \emph{building block} $B_i$ in this context is a graph with vertex set $V(B_i) = \{a_i, b_i, c_i, d_i, e_i, f_i, g_i, h_i\}$ and edge set $E(B_i) = \{a_ib_i, a_ie_i, b_ic_i, c_id_i, c_if_i, d_ie_i, d_ih_i, e_ig_i, f_ig_i\}$, as shown in Figure~\ref{fig:bloco-construcaoGoldBi}. Two blocks $B_j$ and $B_{\ell}$ are joined by a set of \emph{link-edges} defined as $E_{j,\ell} = \{b_ja_{\ell}, g_jf_{\ell}, h_jh_{\ell}\}$, which is illustrated in Figure~\ref{fig:link-graphGold}.

For an odd integer $n \geq 3$, the Goldberg snark $G_n$ is formally defined by $V(G_n) = \bigcup_{i=1}^n V(B_i)$ and $E(G_n) = \left(\bigcup_{i=1}^{n} E(B_i)\right) \cup \left(\bigcup_{j=1}^{n-1} E_{j,j+1}\right) \cup E_{n,1}$. Figure~\ref{fig:snarkG3} shows $G_3$, the first member of this family.

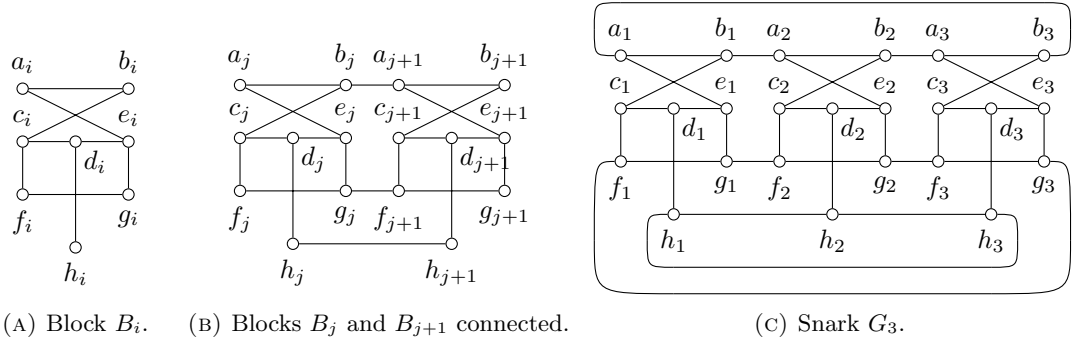
\begin{figure}
     \centering
     \begin{subfigure}[b]{0.15\textwidth}
         \centering
         \begin{tikzpicture}
      [cp/.style={circle,fill=white,draw,minimum size=1pt,inner sep=1.5pt,font=\scriptsize},scale=.7]
      \node[cp,label=above:{$a_i$}] (v15) at (15,0) {};
      \node[cp,label=above:{$b_i$}] (v25) at (17,0) {};
      \node[cp,label=above:{$c_i$}] (v35) at (15,-1) {};
      \node[cp,label={[label distance=-.1cm]330:{$d_i$}}] (v45) at (16,-1) {};
      \node[cp,label=above:{$e_i$}] (v55) at (17,-1) {};
      \node[cp,label=below:{$f_i$}] (v65) at (15,-2) {};
      \node[cp,label=below:{$g_i$}] (v75) at (17,-2) {};
      \node[cp,label=below:{$h_i$}] (v85) at (16,-3) {};
      
      \draw (v15) -- (v25) -- (v35) -- (v45) -- (v55) -- (v75) -- (v65) -- (v35)  (v15) -- (v55) (v45) -- (v85);

    \end{tikzpicture}
    \caption{Block $B_i$.}
    \label{fig:bloco-construcaoGoldBi}
     \end{subfigure}
     \begin{subfigure}[b]{0.35\textwidth}
         \centering
         \begin{tikzpicture}
      [cp/.style={circle,fill=white,draw,minimum size=1pt,inner sep=1.5pt,font=\scriptsize},scale=.7]
      \node[cp,label=above:{$a_{j}$}] (v1) at (0,0) {};
      \node[cp,label=above:{$b_{j}$}] (v2) at (2,0) {};
      \node[cp,label=above:{$c_{j}$}] (v3) at (0,-1) {};
      \node[cp,label={[label distance=-.1cm]330:{$d_{j}$}}] (v4) at (1,-1) {};
      \node[cp,label=above:{$e_{j}$}] (v5) at (2,-1) {};
      \node[cp,label=below:{$f_{j}$}] (v6) at (0,-2) {};
      \node[cp,label=below:{$g_{j}$}] (v7) at (2,-2) {};
      \node[cp,label=below:{$h_{j}$}] (v8) at (1,-3) {};

      \node[cp,label=above:{$a_{j+1}$}] (v11) at (3,0) {};
      \node[cp,label=above:{$b_{j+1}$}] (v21) at (5,0) {};
      \node[cp,label=above:{$c_{j+1}$}] (v31) at (3,-1) {};
      \node[cp,label={[label distance=-.1cm]330:{$d_{j+1}$}}] (v41) at (4,-1) {};
      \node[cp,label=above:{$e_{j+1}$}] (v51) at (5,-1) {};
      \node[cp,label=below:{$f_{j+1}$}] (v61) at (3,-2) {};
      \node[cp,label=below:{$g_{j+1}$}] (v71) at (5,-2) {};
      \node[cp,label=below:{$h_{j+1}$}] (v81) at (4,-3) {};
      
      \draw (v1) -- (v2) -- (v3) -- (v4) -- (v5) -- (v7) -- (v6) -- (v3)  (v1) -- (v5) (v4) -- (v8);
      \draw (v11) -- (v21) -- (v31) -- (v41) -- (v51) -- (v71) -- (v61) -- (v31)  (v11) -- (v51) (v41) -- (v81);
      \draw (v2)--(v11) (v7)--(v61) (v8)--(v81);

    \end{tikzpicture}
    \caption{Blocks $B_j$ and $B_{j+1}$ connected.}
    \label{fig:link-graphGold}
     \end{subfigure}
     \begin{subfigure}[b]{0.4\textwidth}
    \centering
    \begin{tikzpicture}
      [cp/.style={circle,fill=white,draw,minimum size=1pt,inner sep=1.5pt,font=\scriptsize},scale=.7]
      \node[cp,label=above:{$a_1$}] (v1) at (0,0) {};
      \node[cp,label=above:{$b_1$}] (v2) at (2,0) {};
      \node[cp,label=above:{$c_1$}] (v3) at (0,-1) {};
      \node[cp,label={[label distance=-.1cm]330:{$d_1$}}] (v4) at (1,-1) {};
      \node[cp,label=above:{$e_1$}] (v5) at (2,-1) {};
      \node[cp,label=below:{$f_1$}] (v6) at (0,-2) {};
      \node[cp,label=below:{$g_1$}] (v7) at (2,-2) {};
      \node[cp,label=below:{$h_1$}] (v8) at (1,-3) {};

      \node[cp,label=above:{$a_2$}] (v11) at (3,0) {};
      \node[cp,label=above:{$b_2$}] (v21) at (5,0) {};
      \node[cp,label=above:{$c_2$}] (v31) at (3,-1) {};
      \node[cp,label={[label distance=-.1cm]330:{$d_2$}}] (v41) at (4,-1) {};
      \node[cp,label=above:{$e_2$}] (v51) at (5,-1) {};
      \node[cp,label=below:{$f_2$}] (v61) at (3,-2) {};
      \node[cp,label=below:{$g_2$}] (v71) at (5,-2) {};
      \node[cp,label=below:{$h_2$}] (v81) at (4,-3) {};

      \node[cp,label=above:{$a_3$}] (v12) at (6,0) {};
      \node[cp,label=above:{$b_3$}] (v22) at (8,0) {};
      \node[cp,label=above:{$c_3$}] (v32) at (6,-1) {};
      \node[cp,label={[label distance=-.1cm]330:{$d_3$}}] (v42) at (7,-1) {};
      \node[cp,label=above:{$e_3$}] (v52) at (8,-1) {};
      \node[cp,label=below:{$f_3$}] (v62) at (6,-2) {};
      \node[cp,label=below:{$g_3$}] (v72) at (8,-2) {};
      \node[cp,label=below:{$h_3$}] (v82) at (7,-3) {};

      \draw (v1) -- (v2) -- (v3) -- (v4) -- (v5) -- (v7) -- (v6) -- (v3)  (v1) -- (v5) (v4) -- (v8);
      \draw (v11) -- (v21) -- (v31) -- (v41) -- (v51) -- (v71) -- (v61) -- (v31)  (v11) -- (v51) (v41) -- (v81);
      \draw (v12) -- (v22) -- (v32) -- (v42) -- (v52) -- (v72) -- (v62) -- (v32)  (v12) -- (v52) (v42) -- (v82);
      \draw (v2)--(v11) (v7)--(v61) (v8)--(v81) (v21)--(v12) (v71)--(v62) (v81)--(v82);

      \draw (v8) .. controls (.5,-3) .. (.5,-3.5);
      \draw (.5,-3.5) .. controls (.5,-4) .. (1,-4);
      \draw (v82) .. controls (7.5,-3) .. (7.5,-3.5);
      \draw (7.5,-3.5) .. controls (7.5,-4) .. (7,-4);
      \draw (1,-4) -- (7,-4);

      \draw (v6) .. controls (-0.5,-2) .. (-0.5,-3.5);
      \draw (-0.5,-3.5) .. controls (-0.5,-4.5) .. (1,-4.5);
      \draw (v72) .. controls (8.5,-2) .. (8.5,-3.5);
      \draw (8.5,-3.5) .. controls (8.5,-4.5) .. (7,-4.5);
      \draw (1,-4.5) -- (7,-4.5);

      \draw (v1) .. controls (-0.5,0) .. (-0.5,.5);
      \draw (-0.5,.5) .. controls (-0.5,1) .. (0,1);
      \draw (v22) .. controls (8.5,0) .. (8.5,.5);
      \draw (8.5,.5) .. controls (8.5,1) .. (8,1);
      \draw (0,1) -- (8,1);
    \end{tikzpicture}
    \caption{Snark $G_3$.}
    \label{fig:snarkG3}
     \end{subfigure}
        \caption{Stages of construction of a Goldberg snark.}
        \label{fig:GolbergBlocks}
\end{figure}

Theorem~\ref{the:goldberg-snark-proof} determines the exact value of the vertex cover number of Goldberg snarks. Its proof relies on the following auxiliary lemma.

\begin{lemma}
\label{lemma:goldberg_cycle}
Let \(G\) be a graph, \(C\) an induced cycle in \(G\) of length \(n\), and \(S\) a minimum vertex cover of \(G\). Then
$|V(C) \cap S| \geq \lceil n/2 \rceil$.
\end{lemma}

\begin{proof}
Suppose, for the sake of contradiction, that $|V(C) \cap S| \leq \lceil n/2 \rceil - 1$. Since $C$ is an induced $n$-cycle, its $n$ edges can only be covered by vertices in $V(C)$. Each vertex in $V(C) \cap S$ covers exactly two edges of $C$. Thus, the number of edges of $C$ covered by $S$ is at most $2 ( \lceil n/2 \rceil - 1 ) < n$. This leaves at least one edge of $C$ uncovered, contradicting the fact that $S$ is a vertex cover. Therefore, $|V(C) \cap S| \geq \lceil n/2 \rceil$.
\end{proof}

\begin{theorem}
\label{the:goldberg-snark-proof}
Let $G_n$ be a Goldberg snark. Then, $\tau(G_n) = \lceil 9n/2 \rceil$.
\end{theorem}

\begin{proof}
    We first establish the upper bound $\tau(G_n) \leq \lceil 9n/2 \rceil$. Define a vertex subset $S \subseteq V(G_n)$ as:
    $$S = \left\{b_{j},d_{j},e_{j},f_{j} \mid 1 \leq j \leq n\right\} \cup \left\{h_{2i+1} \mid 0 \leq i \leq \frac{n-1}{2}\right\}$$
    
    The cardinality of $S$ is exactly $4n + \lceil n/2 \rceil = \lceil 9n/2 \rceil$. Figure~\ref{fig:G9} shows the snark $G_9$ with such a cover $S$.
    
    To verify that $S$ is a valid vertex cover, we analyze the edges of $G_n$. In each block $B_i$, the subset of vertices $\{b_i, d_i, e_i, f_i\}$ covers all edges in the set $E(B_i)$, since the remaining vertices $a_i$, $c_i$, and $g_i$ only have neighbors within this subset, and the pendant edge $d_ih_i$ is covered by $d_i$. The link-edges connecting consecutive blocks, $b_ia_{i+1}$ and $g_if_{i+1}$, are covered because all $b_j$ and $f_j$ vertices belong to $S$. Finally, the edges $h_ih_{i+1}$, are covered by the selection of odd-indexed $h_j$ vertices. Thus, $S$ is a vertex cover of $G_n$.

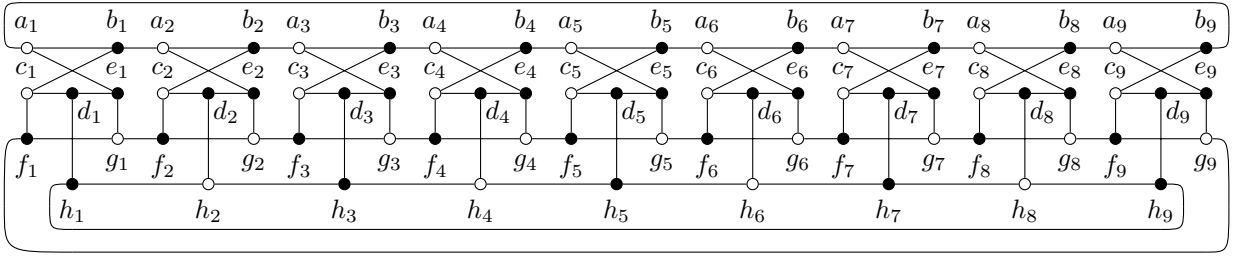
\begin{figure}
    \centering
    \begin{tikzpicture}
[cp/.style={circle,fill=white,draw,minimum size=1pt,inner sep=1.5pt,font=\scriptsize},
 cpblack/.style={circle,fill=black,draw,minimum size=1pt,inner sep=1.5pt,font=\scriptsize},
 scale=.6]

\foreach \i\j in {0/1,1/2,2/3,3/4,4/5,5/6,6/7,7/8,8/9}{
    
    \node[cp,label=above:{$a_\j$}] (v1\i) at (0+3*\i,0) {};
    
    \node[cpblack,label=above:{$b_\j$}] (v2\i) at (2+3*\i,0) {};
    
    \node[cp,label=above:{$c_\j$}] (v3\i) at (0+3*\i,-1) {};
    
    \node[cpblack,label={[label distance=-.15cm]330:{$d_\j$}}] (v4\i) at (1+3*\i,-1) {};
    
    \node[cpblack,label=above:{$e_\j$}] (v5\i) at (2+3*\i,-1) {};
    
    \node[cpblack,label=below:{$f_\j$}] (v6\i) at (0+3*\i,-2) {};
    
    \node[cp,label=below:{$g_\j$}] (v7\i) at (2+3*\i,-2) {};
    
    \pgfmathparse{mod(\j,2)==1 ? 1 : 0}
    \ifnum\pgfmathresult=1
        \node[cpblack,label=below:{$h_\j$}] (v8\i) at (1+3*\i,-3) {};
    \else
        \node[cp,label=below:{$h_\j$}] (v8\i) at (1+3*\i,-3) {};
    \fi
}

\foreach \i in {0,1,2,3,4,5,6,7,8}{
  \draw (v1\i) -- (v2\i) -- (v3\i) -- (v4\i) -- (v5\i) -- (v7\i) -- (v6\i) -- (v3\i)
        (v1\i) -- (v5\i) 
        (v4\i) -- (v8\i);
}

\foreach \i\j in {0/1,1/2,3/4,4/5,5/6,6/7,7/8} 
{
  \draw (v2\i) -- (v1\j);
  \draw (v7\i) -- (v6\j);
  \draw (v8\i) -- (v8\j);
}
\draw (v22) -- (v13);
\draw (v72) -- (v63);
\draw (v82) -- (v83);

\draw (v80) .. controls (.5,-3) .. (.5,-3.5);
\draw (.5,-3.5) .. controls (.5,-4) .. (1,-4);
\draw (v88) .. controls (25.5,-3) .. (25.5,-3.5);
\draw (25.5,-3.5) .. controls (25.5,-4) .. (25,-4);
\draw (1,-4) -- (25,-4);

\draw (v60) .. controls (-0.5,-2) .. (-0.5,-3.5);
\draw (-0.5,-3.5) .. controls (-0.5,-4.5) .. (1,-4.5);
\draw (v78) .. controls (26.5,-2) .. (26.5,-3.5);
\draw (26.5,-3.5) .. controls (26.5,-4.5) .. (26,-4.5);
\draw (1,-4.5) -- (26,-4.5);

\draw (v10) .. controls (-0.5,0) .. (-0.5,.5);
\draw (-0.5,.5) .. controls (-0.5,1) .. (0,1);
\draw (v28) .. controls (26.5,0) .. (26.5,.5);
\draw (26.5,.5) .. controls (26.5,1) .. (26,1);
\draw (0,1) -- (26,1);

\end{tikzpicture}
    \caption{Snark $G_9$ with a minimum vertex cover indicated by the vertices with black color.}
    \label{fig:G9}
\end{figure}

    Next, we prove the lower bound $\tau(G_n) \geq \lceil 9n/2 \rceil$. Let $S$ be a minimum vertex cover of $G_n$. Also, let $C_n \subset G_n$ be the cycle of length $n$ induced by the vertices $\{h_1, h_2, \ldots, h_n\} \subset E(G_n)$. By Lemma~\ref{lemma:goldberg_cycle}, we have $|V(C_n) \cap S| \geq \lceil n/2 \rceil$.
    
    For each $1 \leq i \leq n$, let $H_i \subset B_i$ be the subgraph induced by $\{a_i, b_i, c_i, d_i, e_i, f_i, g_i\}$, as illustrated in Figure~\ref{fig:subgraph_b1}. Note that each subgraph $H_i$ contains an induced 5-cycle $C_5^i = (a_i,b_i,c_i,d_i,e_i)$. By Lemma~\ref{lemma:goldberg_cycle}, $|V(C_5^i)\cap S| \geq 3$. Furthermore, the edge $f_ig_i \in E(H_i)$ does not incide in any vertex of $C_5^i$. Covering this edge requires at least one additional vertex. Therefore, $|V(H_i) \cap S| \geq 3 + 1 = 4$. 
    Since $C_n$ and all subgraphs $H_1, H_2, \ldots, H_n$ are mutually vertex-disjoint, any vertex cover $S$ must contain at least $\lceil n/2 \rceil$ vertices from $C_n$ and $4$ vertices from each $H_i$. Consequently, 
    $|S| \geq 4n + \lceil n/2 \rceil = \lceil 9n/2 \rceil$.
    Since the lower and upper bounds match, we conclude that $\tau(G_n) = \lceil 9n/2 \rceil$.
\end{proof}

\begin{figure}
\centering
\begin{tikzpicture}
      [cp/.style={circle,fill=white,draw,minimum size=1pt,inner sep=1.5pt,font=\scriptsize},scale=.7]
      \node[cp,label=above:{$a_i$}] (v15) at (15,0) {};
      \node[cp,label=above:{$b_i$}] (v25) at (17,0) {};
      \node[cp,label=above:{$c_i$}] (v35) at (15,-1) {};
      \node[cp,label={[label distance=-.1cm]330:{$d_i$}}] (v45) at (16,-1) {};
      \node[cp,label=above:{$e_i$}] (v55) at (17,-1) {};
      \node[cp,label=below:{$f_i$}] (v65) at (15,-2) {};
      \node[cp,label=below:{$g_i$}] (v75) at (17,-2) {};
            
      \draw (v15) -- (v25) -- (v35) -- (v45) -- (v55) -- (v75) -- (v65) -- (v35)  (v15) -- (v55) (v45);
    \end{tikzpicture}
\caption{Subgraph \(H_i\).}
\label{fig:subgraph_b1}
\end{figure}
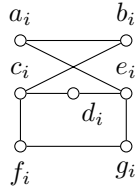

\section{The first family of generalized Blanu\v{s}a snarks}

In 1983, John J.~Watkins~\cite{watkins1983snarks} constructed two infinite families of snarks that are known as the first and the second families of generalized Blanu\v{s}a snarks. In this paper, we focus on the first family. The members of the \emph{first family of generalized Blanu\v{s}a snarks} are constructed from subgraphs called \emph{building blocks}. There are two distinct building blocks in this class, namely the subgraphs $A_1$ and $L_j$, where $j \geq 1$, and they are defined in Figure~\ref{fig:blanusa_blocks}. The degree-2 vertices $a$, $b$, $c$, and $d$ in $A_1$, as well as the degree-2 vertices $x_j$, $y_j$, $w_j$, and $z_j$ in $L_j$, are called \emph{link vertices}.

A generalized Blanuša snark of the first family is denoted by $B_j^1$ and is formed by $j$ blocks $L_1, L_2, \ldots, L_j$ and one single block $A_1$. Thus, $\mathfrak{B}^1 = \{B_1^1, B_2^1, B_3^1, \ldots\}$ denotes the first family of generalized Blanuša snarks.

\begin{figure}
\centering
\begin{subfigure}[b]{0.45\textwidth}
\centering
\begin{tikzpicture}[
    vertex/.style={circle, draw, fill=white, inner sep=2pt, thick},
    edge/.style={thick},
    scale=1.2
]

\node[vertex, label=above:$a$] (a) at (0, 2) {};
\node[vertex, label=left:$e$] (v01) at (0, 1) {};
\node[vertex, label=below:$b$] (b) at (0, 0) {};

\node[vertex, label=above:$f$] (v12) at (1, 2) {};
\node[vertex, label=below:$g$] (v10) at (1, 0) {};

\node[vertex, label=above:$h$] (v22) at (2, 2) {};
\node[vertex, label=right:$i$] (v21) at (2, 1) {};
\node[vertex, label=below:$j$] (v20) at (2, 0) {};

\node[vertex, label=above:$c$] (c) at (3, 2) {};
\node[vertex, label=below:$d$] (d) at (3, 0) {};

\draw[edge] (a) -- (v12) -- (v22) -- (c); 
\draw[edge] (v01) -- (v21);     
\draw[edge] (b) -- (v10) -- (v20) -- (d); 

\draw[edge] (a) -- (v01) -- (b);         
\draw[edge] (v12) -- (v10);    
\draw[edge] (v22) -- (v21) -- (v20);    
\draw[edge] (c) -- (d);                 
\end{tikzpicture}
\caption{Block $A_1$.}
\label{fig:bloco-construcaoGoldA1}
\end{subfigure}
\begin{subfigure}[b]{0.45\textwidth}
\centering
\begin{tikzpicture}[
    vertex/.style={circle, draw, fill=white, inner sep=2pt, thick},
    edge/.style={thick},
    scale=1.2
]

\node[vertex, label=above:{$x_j$}] (x) at (0, 2) {}; 
\node[vertex, label=left:{$u_j$}] (v01) at (0, 1) {};
\node[vertex, label=below:{$y_j$}] (y) at (0, 0) {};

\node[vertex, label=above:{$v_j$}] (v12) at (1, 2) {};
\node[vertex, label=below:{$s_j$}] (v10) at (1, 0) {};

\node[vertex, label=above:{$w_j$}] (w) at (2, 2) {};
\node[vertex, label=right:{$t_j$}] (v21) at (2, 1) {};
\node[vertex, label=below:{$z_j$}] (z) at (2, 0) {};

\draw[edge] (x) -- (v12) -- (w); 
\draw[edge] (v01) -- (v21);        
\draw[edge] (y) -- (v10) -- (z); 

\draw[edge] (x) -- (v01) -- (y);            
\draw[edge] (v12) -- (v10);        
\draw[edge] (w) -- (v21) -- (z);        
\end{tikzpicture}
\caption{Block $L_j$.}
\label{fig:bloco-construcaoGoldLj}
\end{subfigure}
\caption{Building blocks of generalized Blanu\v{s}a snarks.}
\label{fig:blanusa_blocks}
\end{figure}
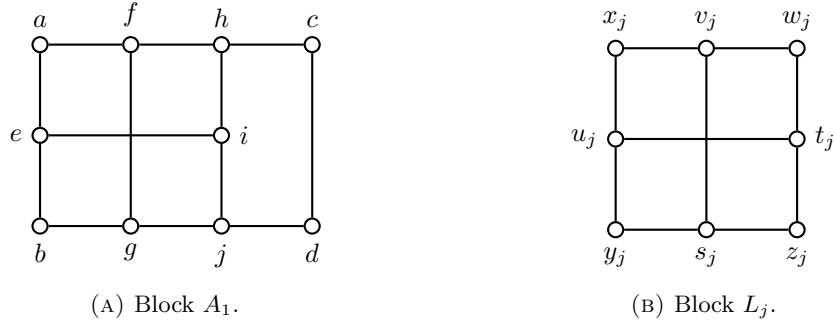

The generalized Blanu\v{s}a snark $B_i^1$ is the graph with vertex set $V(B_i^1) = V(A_1)\cup \left(\bigcup_{k=1}^i V(L_k) \right)$ and edge set $E(B_i^1) = E(A_1)\cup \left(\bigcup_{k=1}^i E(L_k)\right) \cup \left\{ z_jx_{j+1}, w_jy_{j+1} \mid 1\leq j \leq i-1 \right\} \cup \left\{ y_1c, x_1d, z_ia,w_2b \right\}$. As an example, Figure~\ref{fig:blanusa_1_2} illustrates the snark $B_2^1$. The edges whose both endpoints are link-vertices are called \emph{link-edges}.

Lemma~\ref{lemma:blanusa-snark-proof} establishes a lower bound on the vertex cover number of the graphs in the first family of generalized Blanu\v{s}a snarks. To prove this result, we first present the following sequence of lemmas, which will be used in its proof.

\begin{figure}
\centering
\begin{tikzpicture}[
    vertex/.style={circle, draw, fill=white, inner sep=2pt, font=\large, thick},
    edge/.style={thick},
    scale=0.5
]

    \node[vertex, label=left:$a$] (a) at (0, 3) {};
    \node[vertex] (l_mid) at (0, 1.5) {};
    \node[vertex, label=left:$b$] (b) at (0, 0) {};
    
    \node[vertex] (c_top_l) at (2, 3) {};
    \node[vertex] (c_mid_l) at (2, 1.5) {};
    \node[vertex] (c_bot_l) at (2, 0) {};
    
    \node[vertex, label=above:$c$] (c) at (4, 3) {};
    \node[vertex, label=below:$d$] (d) at (4, 0) {};

    \node[vertex] (cross_top) at (1, 3) {}; 
    \node[vertex] (cross_bot) at (1, 0) {}; 
    \draw[edge] (cross_top) -- (cross_bot);

    \draw[edge] (a) -- (cross_top) -- (c_top_l) -- (c);
    \draw[edge] (l_mid) -- (c_mid_l);
    \draw[edge] (b) -- (cross_bot) -- (c_bot_l) -- (d);
    
    \draw[edge] (a) -- (l_mid) -- (b);
    \draw[edge] (c_top_l) -- (c_mid_l) -- (c_bot_l);
    \draw[edge] (c) -- (d);
    
    \begin{scope}[xshift=6cm]
        \node[vertex, label=above:$x_1$] (x1) at (0, 3) {};
        \node[vertex] (m1_l_mid) at (0, 1.5) {};
        \node[vertex, label=below:$y_1$] (y1) at (0, 0) {};
        
        \node[vertex] (m1_c_top) at (2, 3) {};
        \node[vertex] (m1_c_bot) at (2, 0) {};
        
        \node[vertex, label=above:$w_1$] (w1) at (4, 3) {};
        \node[vertex] (m1_r_mid) at (4, 1.5) {}; 
        \node[vertex, label=below:$z_1$] (z1) at (4, 0) {};
        
        \draw[edge] (x1) -- (m1_c_top) -- (w1);
        \draw[edge] (m1_l_mid) -- (m1_r_mid); 
        \draw[edge] (y1) -- (m1_c_bot) -- (z1);
        
        \draw[edge] (x1) -- (m1_l_mid) -- (y1);
        \draw[edge] (m1_c_top) -- (m1_c_bot);
        \draw[edge] (w1) -- (m1_r_mid) -- (z1);
    \end{scope}

    \begin{scope}[xshift=12cm]
        \node[vertex, label=above:$x_2$] (x2) at (0, 3) {};
        \node[vertex] (m2_l_mid) at (0, 1.5) {};
        \node[vertex, label=below:$y_2$] (y2) at (0, 0) {};
        
        \node[vertex] (m2_c_top) at (2, 3) {};
        \node[vertex] (m2_c_bot) at (2, 0) {};
        
        \node[vertex, label=above:$w_2$] (w2) at (4, 3) {};
        \node[vertex] (m2_r_mid) at (4, 1.5) {}; 
        \node[vertex, label=below:$z_2$] (z2) at (4, 0) {};
        
        \draw[edge] (x2) -- (m2_c_top) -- (w2);
        \draw[edge] (m2_l_mid) -- (m2_r_mid); 
        \draw[edge] (y2) -- (m2_c_bot) -- (z2);
        
        \draw[edge] (x2) -- (m2_l_mid) -- (y2);
        \draw[edge] (m2_c_top) -- (m2_c_bot);
        \draw[edge] (w2) -- (m2_r_mid) -- (z2);
    \end{scope}

    \draw[edge] (c) -- (y1);
    \draw[edge] (d) -- (x1);
    
    \draw[edge] (w1) -- (y2);
    \draw[edge] (z1) -- (x2);
    
    \draw[edge, rounded corners=15pt] (a) -- ++(0, 1.5) -- (17.5, 4.5) |- (z2);
    
    \draw[edge, rounded corners=15pt] (b) -- ++(0, -1.5) -- (17.0, -1.5) |- (w2);

\end{tikzpicture}
\caption{The generalized Blanu\v{s}a snarks $B_2^1$.}
\label{fig:blanusa_1_2}
\end{figure}
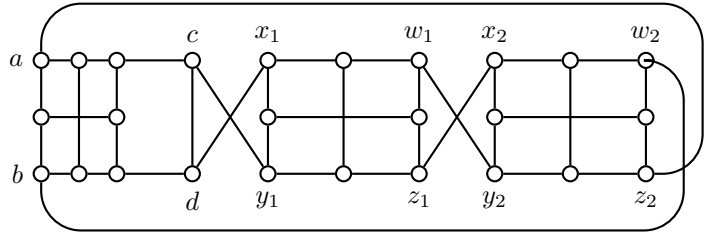

\begin{lemma}
\label{prop:prop_18}
A building block $L_j$ has $\tau(L_j) = 4$. Moreover, it admits a unique vertex cover $S=\{s_j,t_j,u_j,v_j\}$ of cardinality 4. 
\end{lemma}

\begin{proof}
Let $X \subseteq V(L_j)$ be a vertex cover of $L_j$, where $X = \{u_j, v_j, s_j, t_j\}$, as illustrated in Figure~\ref{fig:blanusa_block_cover}. Since $|X| = 4$, we have that $\tau(L_j) \leq |X| = 4$. Since $|V(L_j)|=8$, the maximum size of any matching in $L_j$ is 4. By Proposition~\ref{lemma:minMaxCover}, the cardinality of a maximum matching serves as a lower bound for the minimum vertex cover. Therefore, it follows that $\tau(L_j) \geq 4$. From the established inequalities, we conclude that $\tau(L_j) = 4$.

Let $S$ be a minimum vertex cover of $L_j$. Note that the edges $v_js_j$ and $u_jt_j$ of $L_j$ must be covered by some vertex in $S$, which implies that at least one endpoint of each of these edges must belong to $S$. Consider the drawing of $L_j$, where $x_j, y_j, w_j, z_j$ are positioned at the vertices of a square (see Figure~\ref{fig:blanusa_block_cover}). By the rotational symmetry of the drawing of $L_j$, we may assume, without loss of generality, that $u_j, v_j \in S$, so that the edges $v_js_j$ and $u_jt_j$ are covered. 
However, with only these two vertices in the set $S$, there remain 4 uncovered edges, namely $t_jw_j$, $t_jz_j$, $s_jz_j$, and $s_jy_j$. These edges induce a path $P_5$ with 5 vertices in the graph $L_j$. By Lemma~\ref{lemma:pathsVC}, we have that $\tau(P_5) = 2$. Hence, two vertices of this $P_5$ must be chosen to cover all these remaining edges, and the only way to cover these 4 edges with two vertices is to select the vertices $s_j$ and $t_j$ as part of the cover.
Therefore, $S = \{u_j, v_j, s_j, t_j\}$, and the result follows.
\end{proof}

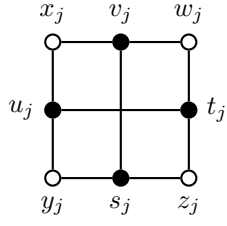
\begin{figure}
\centering
\begin{tikzpicture}[
    vertex/.style={circle, draw, fill=white, inner sep=2pt, thick},
    edge/.style={thick},
    scale=.9
]

\node[vertex, label=above:{$x_j$}] (x) at (0, 2) {};
\node[vertex, fill=black, label=left:{$u_j$}] (v01) at (0, 1) {};
\node[vertex, label=below:{$y_j$}] (y) at (0, 0) {};

\node[vertex, fill=black, label=above:{$v_j$}] (v12) at (1, 2) {};
\node[vertex, fill=black, label=below:{$s_j$}] (v10) at (1, 0) {};

\node[vertex, label=above:{$w_j$}] (w) at (2, 2) {};
\node[vertex, fill=black, label=right:{$t_j$}] (v21) at (2, 1) {};
\node[vertex, label=below:{$z_j$}] (z) at (2, 0) {};

\draw[edge] (x) -- (v12) -- (w);
\draw[edge] (v01) -- (v21);
\draw[edge] (y) -- (v10) -- (z);

\draw[edge] (x) -- (v01) -- (y);
\draw[edge] (v12) -- (v10);
\draw[edge] (w) -- (v21) -- (z);

\end{tikzpicture}
\caption{A building block $L_j$ with its minimum vertex cover.}
\label{fig:blanusa_block_cover}
\end{figure}

\begin{lemma}
\label{lemma:5VC}
If $S$ is a vertex cover of the building block $L_j$ with $|S|=5$, then at least one of the link-vertices of $L_j$ does not belong to $S$.
\end{lemma}

\begin{proof}
Let $S$ be a vertex cover of the building block $L_j$, with $|S|=5$. Assume, for sake of contradiction,  that $S$ contains all the four link-vertices of $L_j$. Thus, one more additional vertex is not sufficient to cover the edges $s_jv_j$ and $t_ju_j$ of the graph $L_j$ since these edges are non-adjacent, contradicting the fact that $S$ covers all edges of $L_j$.  
\end{proof}

\begin{lemma}
\label{prop:blanusa-first-block}
The building block $A_1$ of a generalized Blanu\v{s}a snark $B_i^1$ has $\tau(A_1) = 6$. 
\end{lemma}

\begin{proof}
It can be checked by inspection of Figure~\ref{fig:A1VC} that $A_1$ has a vertex cover with cardinality 6. Therefore, $\tau(A_1)\leq 6$. Next, we prove that $\tau(A_1)\geq 6$. Given \(V(A_1) = \{a, b, c, d, e, f, g, h, i, j\}\), we partition $V(A_1)$ into two sets $V_l = \{a, e, b, g, f\}$ and $V_r = \{c, d, j, i, h\}$ such that these vertex sets induce two vertex-disjoint subgraphs, $G[V_l]$ and $G[V_r]$, both isomorphic to a cycle $C_5$. By Proposition \ref{lemma:cyclesVC}, $\tau(C_5) = \lceil 5/2 \rceil = 3$. Since $G[V_l]$ and $G[V_r]$ are vertex-disjoint subgraphs of $A_1$, by Lemma~\ref{lemma:VCsubgraph}, we have $\tau(A_1) \ge \tau(C_5) + \tau(C_5) = 3 + 3 = 6$. Since $\tau(A_1) \ge 6$ and $\tau(A_1) \le 6$, we conclude that $\tau(A_1) = 6$.
\end{proof}

\tikzset{
    vertex/.style={circle, draw, fill=white, inner sep=2pt, font=\normalsize, thick},
    covered/.style={vertex, fill=black},
    edge/.style={thick}
}

\begin{figure}
\centering
\begin{tikzpicture}[
    vertex/.style={circle, draw, fill=white, inner sep=2pt, thick},
    edge/.style={thick},
    scale=.9
]

\node[covered, label=above:$a$] (a) at (0, 2) {};
\node[vertex, label=left:$e$] (v01) at (0, 1) {};
\node[covered, label=below:$b$] (b) at (0, 0) {};

\node[covered, label=above:$f$] (v12) at (1, 2) {};
\node[vertex, label=below:$g$] (v10) at (1, 0) {};

\node[vertex, label=above:$h$] (v22) at (2, 2) {};
\node[covered, label=right:$i$] (v21) at (2, 1) {};
\node[covered, label=below:$j$] (v20) at (2, 0) {};

\node[covered, label=above:$c$] (c) at (3, 2) {};
\node[vertex, label=below:$d$] (d) at (3, 0) {};

\draw[edge] (a) -- (v12) -- (v22) -- (c); 
\draw[edge] (v01) -- (v21);     
\draw[edge] (b) -- (v10) -- (v20) -- (d); 

\draw[edge] (a) -- (v01) -- (b);         
\draw[edge] (v12) -- (v10);    
\draw[edge] (v22) -- (v21) -- (v20);    
\draw[edge] (c) -- (d);                 
\end{tikzpicture}
\caption{Building block $A_1$ with a minimum vertex cover.}
\label{fig:A1VC}
\end{figure}
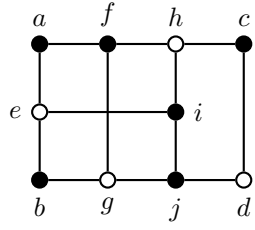

\begin{lemma}
\label{prop:prop_19}
If $G$ is the graph formed by the disjoint union of three building blocks $L_j, L_{j+1}$, and $L_{j+2}$ with the addition of the following link-edges $\{w_jy_{j+1}, z_jx_{j+1}, w_{j+1}y_{j+2}, z_{j+1}x_{j+2}\}$, then $\tau(G) = 14$. Moreover, if $G$ has a minimum vertex cover $S$ such that the vertices $x_j$ and $y_j$ belong to $S$, then the vertices $w_{j+2}$ and $z_{j+2}$ do not belong to $S$.
\end{lemma}

\begin{proof}
Note that the graph $G$, illustrated in Figure \ref{fig:blanusa_three_blocks321}, has three vertex-disjoint subgraphs, namely the blocks $L_j$, $L_{j+1}$, and $L_{j+2}$. A simple lower bound for the vertex cover number of $G$ is $\tau(G) \geq \tau(L_j)+\tau(L_{j+1})+\tau(L_{j+2})$. By Lemma~\ref{prop:prop_18}, we have that $\tau(G) \geq \tau(L_j)+\tau(L_{j+1})+\tau(L_{j+2}) = 4+4+4 = 12$.

Suppose that $G$ has a vertex cover $S$ with cardinality 12. By Lemma~\ref{prop:prop_18}, the only way to obtain such a cover would be by covering each of the subgraphs with the cover shown in Figure~\ref{fig:blanusa_block_cover}, since it is unique. However, note that such a cover would not cover the link-edges between these blocks, which would be a contradiction. Therefore, it is not possible to have a vertex cover $S$ of cardinality 12. Hence, we obtain that $\tau(G) \geq 13$.

Thus, suppose that $G$ has a vertex cover $S$ with cardinality 13. By Lemma~\ref{prop:prop_18}, exactly two of the blocks $L_j$, $L_{j+1}$, and $L_{j+2}$ must have 4 of their vertices in the cover, and exactly one of these blocks must have 5 of its vertices in the cover. However, since two of these blocks are covered by the cover shown in Figure~\ref{fig:blanusa_block_cover} and the other block is covered by a cover with cardinality $5$, it follows by Lemma~\ref{lemma:5VC} that at least one of the link-edges connecting these blocks would not be covered by any vertex in the cover, which is a contradiction. Therefore, it is not possible to have a vertex cover $S$ of cardinality 13, and thus we conclude that $\tau(G)\geq 14$.

Figure \ref{fig:blanusa_three_blocks321} presents one of the possible minimum vertex covers for the graph $G$, thus proving that $\tau(G)\leq 14$. All these facts imply that $\tau(G) = 14$.

Next, we prove that, if $G$ has a minimum vertex cover $S$ such that the vertices $x_j$ and $y_j$ belong to $S$, then the vertices $w_{j+2}$ and $z_{j+2}$ do not belong to $S$.

Suppose that $G$ has a minimum vertex cover $S$ such that $x_j, y_j \in S$. Since $|S|=14$ and each block $L_k$ is such that $|S\cap L_k|\geq 4$ (see Lemma~\ref{prop:blanusa-first-block}), we have that one block of $G$ contains 4 vertices of $S$ and each one of the two remaining blocks of $G$ contains 5 vertices of $S$. By Lemma~\ref{prop:blanusa-first-block}, the unique way for a block $L_k$ to have $|S\cap L_k|=4$ is to choose its vertices $u_k,v_k,s_k,t_k$ to belong to $S$. Since $x_j,y_j \in S$, we obtain that the $|S\cap L_j|=5$.

The subgraph $H = G[V(L_j)\backslash \{u_j,x_j,y_j\}] \subset L_j$ is isomorphic to a 5-cycle. By Lemma~\ref{lemma:cyclesVC}, exactly 3 vertices from $H$ must be in $S$ in order to cover $H$. Since the edge $u_jt_j \in E(L_j)$ must be covered and it incides in the vertex $t_j \in V(H)$, we have that $t_j\in S$. Since the edge $s_jv_j \in E(H)$  must be covered, either $v_j$ or $s_j$ must belong to $S$. By simmetry of $G$, without loss of generality, we choose $s_j \in S$. Since the edge $v_jw_j$ must be covered, we have that either $w_j \in S$ or $v_j\in S$, w.l.o.g.~say that $w_j \in S$. Since the edge $z_jx_{j+1}$ must be covered and 5 vertices of $L_j$ were already choosed to be in $S$, we have that $x_{j+1}\in S$. 

By Lemma~\ref{prop:blanusa-first-block}, the unique way for a block $L_k$ to have $|S\cap L_k|=4$ is to choose its vertices $u_k,v_k,s_k,t_k$ to belong to $S$. Since $x_{j+1} \in S$, we obtain that the $|S\cap L_{j+1}|=5$. Since $|S|=14$, $|S\cap L_j|=5$ and $|S\cap L_{j+1}|=5$, we conclude that $|S\cap L_{j+2}|=4$. By Lemma~\ref{prop:blanusa-first-block}, the unique way for the block $L_{j+2}$ to have $|S\cap L_{j+2}|=4$ is to choose its vertices $u_{j+2},v_{j+2},s_{j+2},t_{j+2}$ to belong to $S$. This implies that $w_{j+2}$ and $z_{j+2}$ do not belong to $S$.
\end{proof}

\begin{lemma}
\label{lemma:Fgraph}
If $F$ is the graph obtained from the disjoint union of the building blocks $A_1$ and $L_1$ together with the addition of the link-edges $y_1c$ and $x_1d$,  then  $\tau(F)=10$. Moreover, $F$ admits a unique vertex cover of cardinality $10$ and the link-vertices $a,b,w_1,z_1$ do not belong to this cover. 
\end{lemma}

\begin{proof}
By Lemmas~\ref{lemma:VCsubgraph},~\ref{prop:prop_18} and~\ref{prop:blanusa-first-block}, we have that $\tau(F) \geq \tau(A_1)+\tau(L_1) = 6+4=10$. Figure~\ref{fig:blanusa_1abc} shows $F$ with a vertex cover of cardinality 10. Therefore, $\tau(F)=10$. 

Let $S$ be a minimum vertex cover of $F$. Next, we prove that $S$ is unique and that $a,b,w_1,z_1 \not\in S$. By Lemmas~\ref{lemma:VCsubgraph},~\ref{prop:prop_18} and~\ref{prop:blanusa-first-block}, we have that $|S\cap L_1| = 4$ and $|S\cap A_1| = 6$. Moreover, also by Lemma~\ref{prop:prop_18}, we have that $w_1,x_1,y_1,z_1 \not\in S$. Since the link-edges $cy_1$ and $dx_1$ must be covered by $S$ and $y_1,x_1 \not\in S$, this implies that $c,d \in S$. Note that the subgraph $H \subset A_1$ induced by the vertices in the set $V(A_1)\backslash\{c,d\}$ is isomorphic to the block $L_j$. Since 4 vertices of $S$ cover the vertices of $L_1$ and the vertices $c$ and $d$ are also in $S$, it remains 4 vertices of $S$ to cover the graph $H$ that is isomorphic to $L_j$. By Lemma~\ref{prop:prop_18}, we obtain that there is a unique way to cover $H$ with 4 vertices that is by choosing the vertices $e,f,g,i$ to put in $S$. Therefore, $F$ has a unique vertex cover $S=\{e,f,g,i,c,d,u_1,v_1,s_1,t_1\}$ and $a,b,w_1,z_1 \not\in S$.
\end{proof}

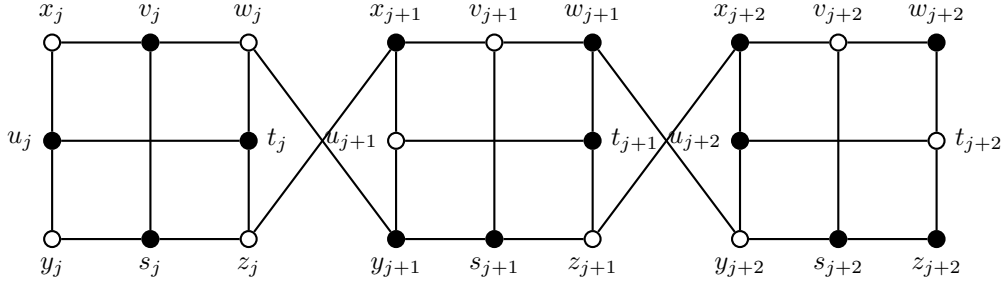
\begin{figure}
\centering
\begin{tikzpicture}[
    scale=1.3,
    vertex/.style={circle, draw, fill=white, inner sep=0pt, minimum size=6pt, thick},
    covered/.style={vertex, fill=black},
    edge/.style={thick}
]

    \node[covered, label=left:$u_j$] (b1_u) at (0, 1) {}; 
    \node[covered, label=above:$v_j$] (b1_v) at (1, 2) {};
    \node[covered, label=below:$s_j$] (b1_s) at (1, 0) {};
    \node[covered, label=right:$t_j$] (b1_t) at (2, 1) {}; 

    \node[vertex, label=above:$x_j$] (b1_x) at (0, 2) {};
    \node[vertex, label=above:$w_j$] (b1_w) at (2, 2) {}; 
    \node[vertex, label=below:$y_j$] (b1_y) at (0, 0) {}; 
    \node[vertex, label=below:$z_j$] (b1_z) at (2, 0) {}; 

    \draw[edge] (b1_x) -- (b1_v) -- (b1_w);
    \draw[edge] (b1_y) -- (b1_s) -- (b1_z);
    \draw[edge] (b1_x) -- (b1_u) -- (b1_y);
    \draw[edge] (b1_w) -- (b1_t) -- (b1_z);
    \draw[edge] (b1_v) -- (b1_s);
    \draw[edge] (b1_u) -- (b1_t); 

    \begin{scope}[shift={(3.5,0)}]
        \node[covered, label=above:$x_{j+1}$] (b2_x) at (0, 2) {}; 
        \node[covered, label=below:$y_{j+1}$] (b2_y) at (0, 0) {};
        
        \node[vertex, label=left:$u_{j+1}$] (b2_u) at (0, 1) {}; 
        \node[vertex, label=above:$v_{j+1}$] (b2_v) at (1, 2) {}; 
        \node[covered, label=below:$s_{j+1}$] (b2_s) at (1, 0) {}; 
        \node[covered, label=right:$t_{j+1}$] (b2_t) at (2, 1) {};
        \node[covered, label=above:$w_{j+1}$] (b2_w) at (2, 2) {}; 
        \node[vertex, label=below:$z_{j+1}$] (b2_z) at (2, 0) {};
        
        \draw[edge] (b2_x) -- (b2_v) -- (b2_w);
        \draw[edge] (b2_y) -- (b2_s) -- (b2_z);
        \draw[edge] (b2_x) -- (b2_u) -- (b2_y);
        \draw[edge] (b2_w) -- (b2_t) -- (b2_z);
        \draw[edge] (b2_v) -- (b2_s);
        \draw[edge] (b2_u) -- (b2_t);
    \end{scope}

    \begin{scope}[shift={(7.0,0)}]
        \node[covered, label=above:$x_{j+2}$] (b3_x) at (0, 2) {}; 
        \node[vertex, label=below:$y_{j+2}$] (b3_y) at (0, 0) {}; 
        
        \node[covered, label=left:$u_{j+2}$] (b3_u) at (0, 1) {}; 
        \node[vertex, label=above:$v_{j+2}$] (b3_v) at (1, 2) {}; 
        \node[covered, label=below:$s_{j+2}$] (b3_s) at (1, 0) {}; 
        \node[vertex, label=right:$t_{j+2}$] (b3_t) at (2, 1) {}; 
        
        \node[covered, label=above:$w_{j+2}$] (b3_w) at (2, 2) {};
        \node[covered, label=below:$z_{j+2}$] (b3_z) at (2, 0) {}; 
        
        \draw[edge] (b3_x) -- (b3_v) -- (b3_w);
        \draw[edge] (b3_y) -- (b3_s) -- (b3_z);
        \draw[edge] (b3_x) -- (b3_u) -- (b3_y);
        \draw[edge] (b3_w) -- (b3_t) -- (b3_z);
        \draw[edge] (b3_v) -- (b3_s);
        \draw[edge] (b3_u) -- (b3_t);
    \end{scope}

    \draw[edge] (b1_w) -- (b2_y);
    \draw[edge] (b1_z) -- (b2_x);

    \draw[edge] (b2_w) -- (b3_y);
    \draw[edge] (b2_z) -- (b3_x);
\end{tikzpicture}
\caption{Graph $G$ with a minimum vertex cover indicated by the vertices in black color.}
\label{fig:blanusa_three_blocks321}
\end{figure}

Now, we are ready to present the lower bound on $\tau(B_i^1)$.

\begin{lemma}
\label{lemma:blanusa-snark-proof}
$\tau(B_i^1) \geq \left\lceil \frac{14i + 17}{3} \right\rceil$ for all $i \geq 1$.
\end{lemma}

\begin{proof}
Let $i\geq 1$ be an integer and let $B_i^1$ be a generalized Blanu\v{s}a snark. Recall that $B_i^1$ is formed by building blocks $A_1$ and $L_1, L_2, \ldots, L_i$. We split the proof into three cases depending on the value of $i$ modulo 3.

\medskip 

\noindent \textit{Case 1:} $i \equiv 0\pmod{3}$.  
For every $j \in \{1,2,\ldots,i/3\}$, let $H_j$ be the subgraph of $B_i^1$ induced by the vertex set $V(L_{3(j-1)+1})\cup V(L_{3(j-1)+2}) \cup V(L_{3(j-1)+3})$. By Lemma~\ref{prop:prop_19}, for every $j \in \{1,2,\ldots,i/3\}$, we have that $\tau(H_j)\geq 14$. Moreover, by Lemma~\ref{prop:blanusa-first-block}, we know that $\tau(A_1)=6$. Note that, by definition, the subgraphs $H_1, H_2, \ldots, H_{i/3}$ are pairwise vertex-disjoint. By Lemma~\ref{lemma:VCsubgraph}, we obtain that $\tau(B_i^1) \geq \tau(A_1)+\sum_{j=1}^{i/3}\tau(H_j) \geq 6 + 14(i/3) = \left\lceil \frac{14i + 17}{3} \right\rceil$.

\medskip

\noindent \textit{Case 2:} $i\equiv 2\pmod{3}$. For every $j \in \{1,2,\ldots,(i-2)/3\}$, let $H_j$ be the subgraph of $B_i^1$ induced by the vertex set $V(L_{3(j-1)+1})\cup V(L_{3(j-1)+2}) \cup V(L_{3(j-1)+3})$. By Lemma~\ref{prop:prop_19}, for every $j \in \{1,2,\ldots,(i-2)/3\}$, we have that $\tau(H_j)\geq 14$. Moreover, by Lemma~\ref{prop:blanusa-first-block}, we know that $\tau(A_1)=6$. Note that the subgraphs $H_1, H_2, \ldots, H_{(i-2)/3}$ are pairwise vertex-disjoint and these subgraphs do not contain the building blocks $A_1$, $L_{i-1}$ and $L_i$. By Lemma~\ref{prop:prop_18}, $\tau(L_{i-1})=4$ and $\tau(L_i)=4$ and the blocks $L_{i-1}$ and $L_i$ have a unique vertex cover of cardinality 4 that is illustrated in Figure~\ref{fig:blanusa_block_cover} (note that the link-vertices of the block do not belong to its vertex cover). However, note that a vertex cover of $B_i^1$, restricted to the subgraph induced by $V(L_{i-1})\cup V(L_i)$, must contain at least 5 vertices of either $L_{i-1}$ or $L_i$ since, otherwise, the link-edges $w_{i-1}y_i$ and $z_{i-1}x_i$ would not be covered. Therefore, we obtain that $\tau(L_i\cup L_{i-1}) \geq 4+5 = 9$. By Lemma~\ref{lemma:VCsubgraph}, we obtain that $\tau(B_i^1) \geq \tau(A_1) + \left[\sum_{j=1}^{(i-2)/3}\tau(H_j)\right] + \tau(G[V(L_i)\cup V(L_{i-1})]) \geq 6 + 14((i-2)/3) + 9 = \left\lceil \frac{14i + 17}{3} \right\rceil$.

\medskip

\noindent \textit{Case 3:} $i\equiv 1\pmod{3}$. For every $j \in \{1,2,\ldots,(i-1)/3\}$, let $H_j$ be the subgraph of $B_i^1$ induced by the vertex set $V(L_{3(j-1)+2})\cup V(L_{3(j-1)+3}) \cup V(L_{3(j-1)+4})$. By Lemma~\ref{prop:prop_19}, for every $j \in \{1,2,\ldots,(i-1)/3\}$, we have that $\tau(H_j)\geq 14$.

Let $S$ be a minimum vertex cover of $B_i^1$. Let $F$ be the subgraph of $B_i^1$ obtained from the disjoint union of the building blocks $A_1$ and $L_1$ along with the link-edges $y_1c$ and $x_1d$. We claim that $|S \cap V(F)| \geq 11$.  

Suppose, for the sake of contradiction, that $|S \cap V(F)| \leq 10$. By Lemma~\ref{lemma:Fgraph}, we have that $\tau(F)=10$. These facts imply that $|S \cap V(F)| = 10$. Moreover, also by Lemma~\ref{lemma:Fgraph}, we know that the link-vertices $a,b,w_1,z_1$ of $F$ do not belong to  $S$. If $i=1$, then we reach a contradiction since the link-edges $az_1$ and $bw_1$ are not covered by $S$. 

Thus, from now on, we may assume that $i \geq 4$. 
Note that, in order to cover the link-edges $z_1x_2$ and $w_1y_2$, the vertices $x_2$ and $y_2$ of the subgraph $H_1$ must belong to $S$. By Lemma~\ref{prop:prop_19}, since $x_2,y_2\in S$, we obtain that $w_4,z_4 \not\in S$. More generally, for each $j \in \{1,2,\ldots,(i-1)/3\}$, in order to cover the link-edges $z_{3(j-1)+1}x_{3(j-1)+2}$ and $w_{3(j-1)+1}y_{3(j-1)+2}$, the vertices $x_{3(j-1)+2}$ and $y_{3(j-1)+2}$ of the subgraph $H_j$ must belong to $S$. By Lemma~\ref{prop:prop_19}, since $x_{3(j-1)+2},y_{3(j-1)+2}\in S$, we obtain that $w_{3(j-1)+4},z_{3(j-1)+4} \not\in S$. When $j = (i-1)/3$, this implies that the link-vertices $w_i$ and $z_i$ do not belong to $S$. However, since the link-vertices $a$ and $b$ also do not belong to $S$, we obtain that the link-edges $az_i$ and $bw_i$ are not covered by $S$, which contradicts the fact that $S$ is a vertex cover of $B_i^1$. Therefore, we conclude that $|S \cap V(F)|\geq 11$.

By Lemma~\ref{lemma:VCsubgraph}, we obtain that $\tau(B_i^1) \geq \tau(F) + \left[\sum_{j=1}^{(i-1)/3}\tau(H_j)\right] \geq 11 + 14((i-1)/3) = \left\lceil \frac{14i + 17}{3} \right\rceil$.
\end{proof}

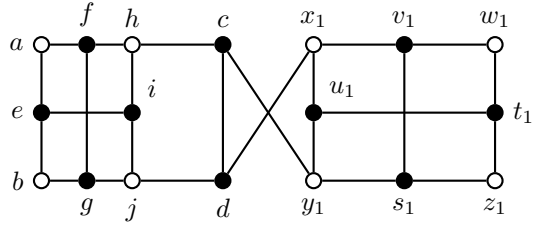
\begin{figure}
\centering
\begin{tikzpicture}[scale=0.6]
    \node[vertex, label=left:$a$] (a) at (0, 3) {};
    \node[covered, label=left:$e$] (l_mid) at (0, 1.5) {}; 
    \node[vertex, label=left:$b$] (b) at (0, 0) {};
    
    \node[covered, label=above:$f$] (cross_top) at (1, 3) {}; 
    \node[covered, label=below:$g$] (cross_bot) at (1, 0) {}; 

    \node[vertex, label=above:$h$] (c_top_l) at (2, 3) {}; 
    \node[covered, label=above right:$i$] (c_mid_l) at (2, 1.5) {};
    \node[vertex, label=below:$j$] (c_bot_l) at (2, 0) {};
    
    \node[covered, label=above:$c$] (c) at (4, 3) {};
    \node[covered, label=below:$d$] (d) at (4, 0) {};

    \draw[edge] (cross_top) -- (cross_bot);
    \draw[edge] (a) -- (cross_top) -- (c_top_l) -- (c);
    \draw[edge] (l_mid) -- (c_mid_l);
    \draw[edge] (b) -- (cross_bot) -- (c_bot_l) -- (d);
    \draw[edge] (a) -- (l_mid) -- (b);
    \draw[edge] (c_top_l) -- (c_mid_l) -- (c_bot_l);
    \draw[edge] (c) -- (d);
    
    \begin{scope}[xshift=6cm]
        \node[vertex, label=above:$x_1$] (x1) at (0, 3) {};  
        \node[covered, label=above right:$u_1$] (u1) at (0, 1.5) {};
        \node[vertex, label=below:$y_1$] (y1) at (0, 0) {};
        
        \node[covered, label=above:$v_1$] (v1) at (2, 3) {}; 
        \node[covered, label=below:$s_1$] (s1) at (2, 0) {};
        
        \node[vertex, label=above:$w_1$] (w1) at (4, 3) {};
        \node[covered, label=right:$t_1$] (t1) at (4, 1.5) {}; 
        \node[vertex, label=below:$z_1$] (z1) at (4, 0) {};
        
        \draw[edge] (x1) -- (v1) -- (w1);
        \draw[edge] (u1) -- (t1); 
        \draw[edge] (y1) -- (s1) -- (z1);
        \draw[edge] (x1) -- (u1) -- (y1);
        \draw[edge] (v1) -- (s1);
        \draw[edge] (w1) -- (t1) -- (z1);
    \end{scope}

    \draw[edge] (c) -- (y1);
    \draw[edge] (d) -- (x1);
\end{tikzpicture}
\caption{Graph $F$ with its minimum vertex cover represented by the black vertices.}
\label{fig:blanusa_1abc}
\end{figure}

The next theorem establishes the vertex cover number for the members of the first family of generalized Blanu\v{s}a snarks.

\begin{theorem}
\label{thm:blanusa-snark-proof-final}
$\tau(B_i^1) = \left\lceil \frac{14i + 17}{3} \right\rceil$ for all $i \geq 1$.
\end{theorem}

\begin{proof}
Let $i\geq 1$ be an integer and let $B_i^1$ be a generalized Blanu\v{s}a snark. By Lemma~\ref{lemma:blanusa-snark-proof}, $\tau(B_i^1) \geq \left\lceil \frac{14i + 17}{3} \right\rceil$ for all $i \geq 1$. So, in order to prove the statement, next we construct a vertex cover $S$ for $B_i^1$ with $|S| = \left\lceil \frac{14i + 17}{3} \right\rceil$.  

We define the set $S = S_0 \cup S_1 \cup S_2 \cup S_4$, where for all $1 \leq j \leq i$, $S_0 = \{ s_j, t_j, w_j, x_j, y_j \mid j \equiv 0 \pmod 3 \}$, $S_1 = \{ s_j, u_j, w_j, x_j, z_j \mid j \equiv 1 \pmod 3 \}$, $S_2 = \{ s_j, t_j, u_j, v_j \mid j \equiv 2 \pmod 3 \}$ and $S_4 = \{a,b,c,f,i,j\}$. 

It can be checked by inspection of Figure~\ref{fig:CVAllblocks} that the edges of the blocks $A_1$ and $L_j$ for all $1\leq j \leq i$ are covered by $S$. Thus, in order to prove that $S$ is a vertex cover, it remains to show that all link-edges are covered by $S$. First, note that the link-edges $cy_1$ and $dx_1$ are covered by vertices $c,x_1 \in S$. The link-edges $az_i$ and $bw_i$ are covered by $a,b \in S$. For two link-edges $w_iy_{i+1}$ and $z_ix_{i+1}$, that connect two consecutive blocks $L_i$ and $L_{i+1}$, we have that: (1) if $i\equiv 1\pmod{3}$, then $w_i,z_i\in S$ ; (2) if $i\equiv 2\pmod{3}$, then $x_{i+1},y_{i+1}\in S$; (3) if $i\equiv 0\pmod{3}$, then $w_{i},x_{i+1}\in S$. Thus, the link-edges that connect two consecutive blocks $L_i$ and $L_{i+1}$ are covered by vertices of $S$. Since all edges of $B_i^1$ were shown to be covered by $S$, we conclude that $S$ is a vertex cover of $B_i^1$.

Next, we calculate the cardinality of $S$. Let $k = i\bmod 3$. For every $j \in \{1,2,\ldots,(i-k)/3\}$, let $H_j$ be the subgraph of $B_i^1$ induced by the vertex set $V(L_{3(j-1)+1})\cup V(L_{3(j-1)+2}) \cup V(L_{3(j-1)+3})$.

The graph $B_i^1$ is formed by one block $A_1$, that is covered by the set $S_4$ previously defined. Also, $B_i^1$ is formed by $(i-k)/3$ vertex-disjoint subgraphs $H_j$, that are covered by vertices from the sets $\{s_{3(j-1)+1}, u_{3(j-1)+1}, w_{3(j-1)+1}, x_{3(j-1)+1}$, $z_{3(j-1)+1}\} \subset S_1$, $\{ s_{3(j-1)+2}, t_{3(j-1)+2}, u_{3(j-1)+2}, v_{3(j-1)+2} \} \subset S_2$ and $\{ s_{3(j-1)+3}, t_{3(j-1)+3}, w_{3(j-1)+3}, x_{3(j-1)+3}$, $y_{3(j-1)+3} \} \subset S_0$. This implies that each subgraph $H_j$ contains 14 vertices of $S$. 

If $k=1$, then $B_i^1$ has one additional block $L_i$ that is covered by the vertices $\{s_i, u_i, w_i, x_i$, $z_i\} \subset S_1$. If $k=2$, then $B_i^1$ has two additional blocks $L_i$ and $L_{i+1}$ that are covered by the vertices $\{s_{i-1}, u_{i-1}, w_{i-1}, x_{i-1}$, $z_{i-1}\} \subset S_1$ and $\{ s_i, t_i, u_i, v_i \} \subset S_2$, respectively. 

Note that $i = 3t+k$ for $t \in \mathbb{Z}$ and $k \in \{0,1,2\}$.

If $k = 0$, then $|S| = |S_4| + 14\cdot \frac{(i-k)}{3} = 6+14\cdot \frac{i}{3} = \frac{14i+18}{3} = \frac{14\cdot 3t+18}{3} = \left\lceil \frac{14\cdot 3t + 17}{3} \right\rceil = \left\lceil \frac{14i + 17}{3} \right\rceil$.

If $k = 1$, then $|S| = |S_4| + 14\cdot \frac{(i-k)}{3} + 5 = 6+14\cdot \frac{i-1}{3}+5 = \frac{14i+19}{3} = \frac{14(3t+1)+19}{3} = \left\lceil \frac{14(3t+1) + 17}{3} \right\rceil = \left\lceil \frac{14i + 17}{3} \right\rceil$.

If $k=2$, then $|S| = |S_4| + 14\cdot \frac{(i-k)}{3} + 9 = 6+14\cdot \frac{i-2}{3}+9 = \frac{14i+17}{3} = \frac{14(3t+2)+17}{3} = \left\lceil \frac{14(3t+2) + 17}{3} \right\rceil = \left\lceil \frac{14i + 17}{3} \right\rceil$.

This concludes the proof.
\end{proof}

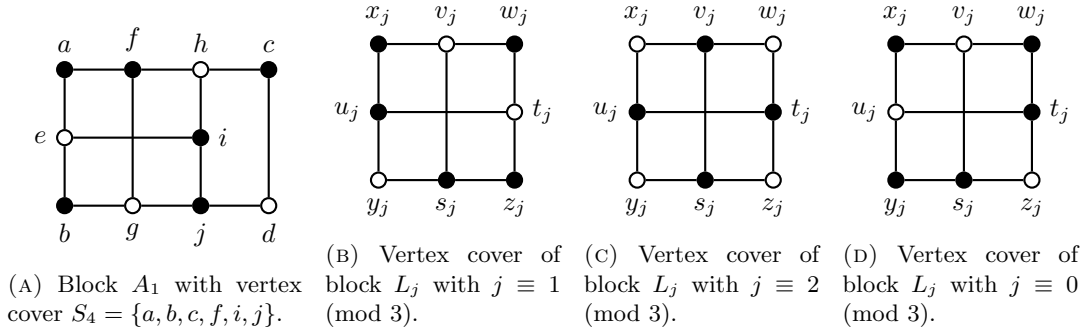
\begin{figure}
\centering
\begin{subfigure}[b]{0.25\textwidth}
\centering
\begin{tikzpicture}[
    vertex/.style={circle, draw, fill=white, inner sep=2pt, thick},
    edge/.style={thick},
    scale=.9
]

\node[covered, label=above:$a$] (a) at (0, 2) {};
\node[vertex, label=left:$e$] (v01) at (0, 1) {};
\node[covered, label=below:$b$] (b) at (0, 0) {};

\node[covered, label=above:$f$] (v12) at (1, 2) {};
\node[vertex, label=below:$g$] (v10) at (1, 0) {};

\node[vertex, label=above:$h$] (v22) at (2, 2) {};
\node[covered, label=right:$i$] (v21) at (2, 1) {};
\node[covered, label=below:$j$] (v20) at (2, 0) {};

\node[covered, label=above:$c$] (c) at (3, 2) {};
\node[vertex, label=below:$d$] (d) at (3, 0) {};

\draw[edge] (a) -- (v12) -- (v22) -- (c); 
\draw[edge] (v01) -- (v21);     
\draw[edge] (b) -- (v10) -- (v20) -- (d); 

\draw[edge] (a) -- (v01) -- (b);         
\draw[edge] (v12) -- (v10);    
\draw[edge] (v22) -- (v21) -- (v20);    
\draw[edge] (c) -- (d);                 
\end{tikzpicture}
\caption{Block $A_1$ with vertex cover $S_4 = \{a,b,c,f,i,j\}$.}
\label{fig:blockCV1}
\end{subfigure}
\hspace{.07cm}
\begin{subfigure}[b]{0.2\textwidth}
\centering
\begin{tikzpicture}[
    vertex/.style={circle, draw, fill=white, inner sep=2pt, thick},
    edge/.style={thick},
    scale=.9
]

\node[covered, label=above:{$x_j$}] (x) at (0, 2) {}; 
\node[covered, label=left:{$u_j$}] (v01) at (0, 1) {};
\node[vertex, label=below:{$y_j$}] (y) at (0, 0) {};

\node[vertex, label=above:{$v_j$}] (v12) at (1, 2) {};
\node[covered, label=below:{$s_j$}] (v10) at (1, 0) {};

\node[covered, label=above:{$w_j$}] (w) at (2, 2) {};
\node[vertex, label=right:{$t_j$}] (v21) at (2, 1) {};
\node[covered, label=below:{$z_j$}] (z) at (2, 0) {};

\draw[edge] (x) -- (v12) -- (w); 
\draw[edge] (v01) -- (v21);        
\draw[edge] (y) -- (v10) -- (z); 

\draw[edge] (x) -- (v01) -- (y);            
\draw[edge] (v12) -- (v10);        
\draw[edge] (w) -- (v21) -- (z);        
\end{tikzpicture}
\caption{Vertex cover of block $L_j$ with $j\equiv 1\pmod{3}$.}
\label{fig:blockCV2}
\end{subfigure}
\hspace{.07cm}
\begin{subfigure}[b]{0.2\textwidth}
\centering
\begin{tikzpicture}[
    vertex/.style={circle, draw, fill=white, inner sep=2pt, thick},
    edge/.style={thick},
    scale=.9
]

\node[vertex, label=above:{$x_j$}] (x) at (0, 2) {}; 
\node[covered, label=left:{$u_j$}] (v01) at (0, 1) {};
\node[vertex, label=below:{$y_j$}] (y) at (0, 0) {};

\node[covered, label=above:{$v_j$}] (v12) at (1, 2) {};
\node[covered, label=below:{$s_j$}] (v10) at (1, 0) {};

\node[vertex, label=above:{$w_j$}] (w) at (2, 2) {};
\node[covered, label=right:{$t_j$}] (v21) at (2, 1) {};
\node[vertex, label=below:{$z_j$}] (z) at (2, 0) {};

\draw[edge] (x) -- (v12) -- (w); 
\draw[edge] (v01) -- (v21);        
\draw[edge] (y) -- (v10) -- (z); 

\draw[edge] (x) -- (v01) -- (y);            
\draw[edge] (v12) -- (v10);        
\draw[edge] (w) -- (v21) -- (z);        
\end{tikzpicture}
\caption{Vertex cover of block $L_j$ with $j\equiv 2\pmod{3}$.}
\label{fig:blockCV3}
\end{subfigure}
\hspace{.07cm}
\begin{subfigure}[b]{0.2\textwidth}
\centering
\begin{tikzpicture}[
    vertex/.style={circle, draw, fill=white, inner sep=2pt, thick},
    edge/.style={thick},
    scale=.9
]

\node[covered, label=above:{$x_j$}] (x) at (0, 2) {}; 
\node[vertex, label=left:{$u_j$}] (v01) at (0, 1) {};
\node[covered, label=below:{$y_j$}] (y) at (0, 0) {};

\node[vertex, label=above:{$v_j$}] (v12) at (1, 2) {};
\node[covered, label=below:{$s_j$}] (v10) at (1, 0) {};

\node[covered, label=above:{$w_j$}] (w) at (2, 2) {};
\node[covered, label=right:{$t_j$}] (v21) at (2, 1) {};
\node[vertex, label=below:{$z_j$}] (z) at (2, 0) {};

\draw[edge] (x) -- (v12) -- (w); 
\draw[edge] (v01) -- (v21);        
\draw[edge] (y) -- (v10) -- (z); 

\draw[edge] (x) -- (v01) -- (y);            
\draw[edge] (v12) -- (v10);        
\draw[edge] (w) -- (v21) -- (z);        
\end{tikzpicture}
\caption{Vertex cover of block $L_j$ with $j\equiv 0\pmod{3}$.}
\label{fig:blockCV4}
\end{subfigure}
\caption{The vertex cover of block $A_1$ and three distinct vertex covers of block $L_j$, depending on the value of $j$ modulo 3. The vertices of the vertex cover are colored in black.}
\label{fig:CVAllblocks}
\end{figure}

\section{Loupekine Snarks}

In 1976, F.~Loupekine proposed another method for constructing infinite families of snarks from pre-existing snarks. This method was originally described by R. Isaacs~\cite{Isaacs1976}. In this section, we present the construction of a specific family of snarks obtained from the method proposed by Loupekine. Next, we define the members of this family.

A \emph{building block}, denoted by $B_i$, is obtained by removing the vertices of any induced path $P_3$ from the Petersen graph, as shown in Figure~\ref{fig:petersen_sub}, which shows the Petersen graph $G$, and Figure~\ref{fig:building_block_sub}, which shows the resulting building block $B_i$.

A Loupekine snark is composed of an odd number $n \geq 3$ of building blocks $B_0, B_1, \ldots, B_{n-1}$ disposed in cyclic order. For each index $i \in \{0, 1, \ldots, n-1\}$, two consecutive blocks $B_i$ and $B_{i+1}$ in this order are joined by a set of \emph{link-edges}, denoted by $E_{i,i+1}$, which are either \emph{parallel edges} $\{s_i r_{i+1}, v_i u_{i+1}\}$ or \emph{crossing  edges} $\{s_i u_{i+1}, v_i r_{i+1}\}$. Figure~\ref{fig:loupekine_direct} shows consecutive blocks connected via parallel  edges, and Figure~\ref{fig:loupekine_crossed} shows consecutive blocks connected via crossing edges.

At this stage of the construction, the graph formed by the $n$ blocks and their link-edges $\left(\bigcup_{i=0}^{n-2} E_{i,i+1}\right) \cup E_{n-1,0}$ is not yet cubic, since each vertex $t_i$ has degree 2. In order to turn it into a cubic graph, the blocks are partitioned into subsets of size two or three. For each subset of two blocks, $B_j$ and $B_l$, we add the edge $t_j t_l$. For each subset of three blocks, $B_j, B_l$, and $B_m$, we introduce a new auxiliary vertex $z_j$ and add the edges $t_j z_j$, $t_l z_j$, and $t_m z_j$. All these new edges are called \emph{upper link-edges}. 

If the partition consists exclusively of blocks with consecutive indices, the resulting graphs belong to the $LP_0$ snark family. If the partition allows arbitrary groupings of two or three blocks, the graphs belong to the $LP_1$ snark family. Examples of both families are shown in Figure~\ref{fig:loupekine_lp0}, illustrating an $LP_0$ snark with five building blocks, and Figure~\ref{fig:loupekine_lp1}, illustrating an $LP_1$ snark with five building blocks.

By definition, adjacent blocks in $LP_0$ and $LP_1$ snarks can be connected by either parallel or crossing edges. In this work, we investigate a specific subclass which we call \emph{parallel-crossed snarks}. A snark in this subclass satisfies the following condition: each block $B_i$ is connected to $B_{i+1}$ via parallel edges if $i \equiv 0, 2 \pmod{4}$, and via crossing edges if $i \equiv 1, 3 \pmod{4}$. We denote a parallel-crossed snark on $n$ building blocks by $L_n$.

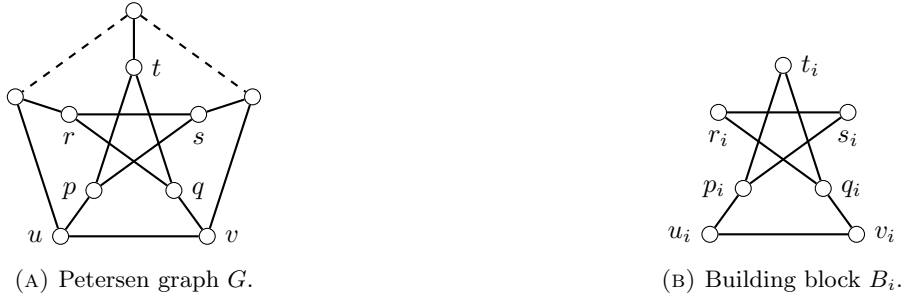
\begin{figure}
\centering
\begin{subfigure}[b]{0.45\textwidth}
\centering
\begin{tikzpicture}[
    vertex/.style={circle, draw, fill=white, inner sep=2pt, minimum size=6pt},
    edge/.style={thick},
    dashed edge/.style={thick, dashed},
    scale=0.75
]
    \def\Rin{1.2}  
    \def\Rout{2.2} 

    \node[vertex, label={right:$t$}] (t) at (90:\Rin) {};
    \node[vertex, label={below:$r$}] (r) at (162:\Rin) {};
    \node[vertex, label={left:$p$}] (p) at (234:\Rin) {};
    \node[vertex, label={right:$q$}] (q) at (306:\Rin) {};
    \node[vertex, label={below:$s$}] (s) at (18:\Rin) {};

    \node[vertex] (top_out) at (90:\Rout) {};
    \node[vertex] (left_out) at (162:\Rout) {};
    \node[vertex, label={left:$u$}] (u) at (234:\Rout) {};
    \node[vertex, label={right:$v$}] (v) at (306:\Rout) {};
    \node[vertex] (right_out) at (18:\Rout) {};

    \draw[edge] (t) -- (p);
    \draw[edge] (p) -- (s);
    \draw[edge] (s) -- (r);
    \draw[edge] (r) -- (q);
    \draw[edge] (q) -- (t);

    \draw[edge] (t) -- (top_out);
    \draw[edge] (r) -- (left_out);
    \draw[edge] (p) -- (u);
    \draw[edge] (q) -- (v);
    \draw[edge] (s) -- (right_out);

    \draw[dashed edge] (top_out) -- (left_out); 
    \draw[dashed edge] (top_out) -- (right_out);
    \draw[edge] (left_out) -- (u);
    \draw[edge] (u) -- (v);
    \draw[edge] (v) -- (right_out);
\end{tikzpicture}
\caption{Petersen graph $G$.}
\label{fig:petersen_sub}
\end{subfigure}
\hfill
\begin{subfigure}[b]{0.45\textwidth}
\centering
\begin{tikzpicture}[
    vertex/.style={circle, draw, fill=white, inner sep=2pt, minimum size=6pt},
    edge/.style={thick},
    scale=0.75
]
    \def\Rin{1.2}  
    \def\Rout{2.2} 

    \node[vertex, label={right:$t_i$}] (t) at (90:\Rin) {};
    \node[vertex, label={below:$r_i$}] (r) at (162:\Rin) {};
    \node[vertex, label={left:$p_i$}] (p) at (234:\Rin) {};
    \node[vertex, label={right:$q_i$}] (q) at (306:\Rin) {};
    \node[vertex, label={below:$s_i$}] (s) at (18:\Rin) {};

    \node[vertex, label={left:$u_i$}] (u) at (234:\Rout) {};
    \node[vertex, label={right:$v_i$}] (v) at (306:\Rout) {};

    \draw[edge] (t) -- (p);
    \draw[edge] (p) -- (s);
    \draw[edge] (s) -- (r);
    \draw[edge] (r) -- (q);
    \draw[edge] (q) -- (t);

    \draw[edge] (p) -- (u);
    \draw[edge] (q) -- (v);
    
    \draw[edge] (u) -- (v);
\end{tikzpicture}
\caption{Building block $B_i$.}
\label{fig:building_block_sub}
\end{subfigure}
\caption{Process of obtaining a building block $B_i$ of the Loupekine snark.}
\label{fig:block_loupekine}
\end{figure}

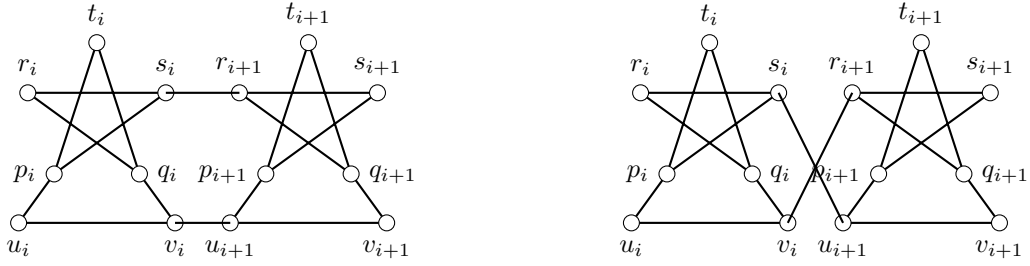
\begin{figure}
\centering

\def\Rin{1.2}   
\def\Rout{2.2}  
\def\BlockSep{3.5}

\newcommand{\DrawBlock}[3]{
    \begin{scope}[xshift=#3]
        \coordinate (#2_t) at (90:\Rin);
        \coordinate (#2_r) at (162:\Rin);
        \coordinate (#2_p) at (234:\Rin);
        \coordinate (#2_q) at (306:\Rin);
        \coordinate (#2_s) at (18:\Rin);
        
        \coordinate (#2_u) at (234:\Rout);
        \coordinate (#2_v) at (306:\Rout);

        \draw[edge] (#2_t) -- (#2_p) -- (#2_s) -- (#2_r) -- (#2_q) -- (#2_t);
        \draw[edge] (#2_p) -- (#2_u);
        \draw[edge] (#2_q) -- (#2_v);
        \draw[edge] (#2_u) -- (#2_v);

        \node[vertex, label={above:$t_{#1}$}] at (#2_t) {};
        \node[vertex, label={above:$r_{#1}$}] at (#2_r) {};
        \node[vertex, label={left:$p_{#1}$}] at (#2_p) {};
        \node[vertex, label={right:$q_{#1}$}] at (#2_q) {};
        \node[vertex, label={above:$s_{#1}$}] at (#2_s) {};
        \node[vertex, label={below:$u_{#1}$}] at (#2_u) {};
        \node[vertex, label={below:$v_{#1}$}] at (#2_v) {};
    \end{scope}
}

\begin{subfigure}[b]{0.48\textwidth}
\centering
\begin{tikzpicture}[
    vertex/.style={circle, draw, fill=white, inner sep=2pt, minimum size=6pt},
    edge/.style={thick},
    label style/.style={font=\small},
    scale=0.8
]
    \DrawBlock{i}{b1}{0cm}
    \DrawBlock{i+1}{b2}{\BlockSep cm}

    \draw[edge] (b1_s) -- (b2_r);
    \draw[edge] (b1_v) -- (b2_u);
\end{tikzpicture}
\caption{Consecutive blocks connected via parallel edges.}
\label{fig:loupekine_direct}
\end{subfigure}
\hfill
\begin{subfigure}[b]{0.48\textwidth}
\centering
\begin{tikzpicture}[
    vertex/.style={circle, draw, fill=white, inner sep=2pt, minimum size=6pt},
    edge/.style={thick},
    label style/.style={font=\small},
    scale=0.8
]
    \DrawBlock{i}{b3}{0cm}
    \DrawBlock{i+1}{b4}{\BlockSep cm}

    \draw[edge] (b3_s) -- (b4_u);
    \draw[edge] (b3_v) -- (b4_r);
\end{tikzpicture}
\caption{Consecutive blocks connected via crossing edges.}
\label{fig:loupekine_crossed}
\end{subfigure}

\caption{Difference between parallel and crossing edges in Loupekine snarks.}
\label{fig:loupekine_edges}
\end{figure}

\begin{figure}
\centering

\def\Rin{0.9}
\def\Rout{1.6}
\def\BlockDist{3.0}
\def\LineDepthInner{-2.0}
\def\LineDepthOuter{-2.4}
\def\OuterMargin{1.2}

\newcommand{\DrawBlock}[2]{
    \begin{scope}[xshift=#1*\BlockDist cm]
        \node[vertex, label={above right:$t_{#1}$}] (#2t#1) at (90:\Rin) {};
        \node[vertex, label={above left:$r_{#1}$}] (#2r#1) at (162:\Rin) {};
        \node[vertex, label={left:$p_{#1}$}] (#2p#1) at (234:\Rin) {};
        \node[vertex, label={right:$q_{#1}$}] (#2q#1) at (306:\Rin) {};
        \node[vertex, label={above right:$s_{#1}$}] (#2s#1) at (18:\Rin) {};
        
        \node[vertex, label={below:$u_{#1}$}] (#2u#1) at (234:\Rout) {};
        \node[vertex, label={below:$v_{#1}$}] (#2v#1) at (306:\Rout) {};

        \draw[edge] (#2t#1) -- (#2p#1) -- (#2s#1) -- (#2r#1) -- (#2q#1) -- (#2t#1);
        \draw[edge] (#2p#1) -- (#2u#1);
        \draw[edge] (#2q#1) -- (#2v#1);
        \draw[edge] (#2u#1) -- (#2v#1);
    \end{scope}
}

\newcommand{\ConnectBlocks}[1]{
    \draw[edge] (#1s0) -- (#1r1); \draw[edge] (#1v0) -- (#1u1);
    \draw[edge] (#1s1) -- (#1u2); \draw[edge] (#1v1) -- (#1r2);
    \draw[edge] (#1s2) -- (#1r3); \draw[edge] (#1v2) -- (#1u3);
    \draw[edge] (#1s3) -- (#1u4); \draw[edge] (#1v3) -- (#1r4);

    \draw[edge, rounded corners=15pt] 
        (#1r0) -- (-\OuterMargin, {0.9*sin(162)})
               -- (-\OuterMargin, \LineDepthOuter)
               -- (4*\BlockDist + \OuterMargin, \LineDepthOuter)
               -- (4*\BlockDist + \OuterMargin, {0.9*sin(18)})
               -- (#1s4);

    \draw[edge, rounded corners=10pt] 
        (#1u0) -- (0, \LineDepthInner)
               -- (4*\BlockDist, \LineDepthInner)
               -- (#1v4);
}

\begin{subfigure}[b]{\textwidth}
\centering
\begin{tikzpicture}[
    vertex/.style={circle, draw, fill=white, inner sep=2pt, minimum size=6pt},
    edge/.style={thick},
    connector/.style={thick},
    scale=0.8
]
    \foreach \i in {0,...,4} { \DrawBlock{\i}{a_} }
    \ConnectBlocks{a_}

    \node[vertex, label={above:$z_0$}] (az0) at (3.0, 2.5) {}; 
    \draw[edge] (a_t0) -- (az0);
    \draw[edge] (a_t1) -- (az0);
    \draw[edge] (a_t2) -- (az0);
    \draw[edge] (a_t3) -- (a_t4);
\end{tikzpicture}
\caption{Example of an $LP_0$ snark with five building blocks.}
\label{fig:loupekine_lp0}
\end{subfigure}

\vspace{0.8cm}

\begin{subfigure}[b]{\textwidth}
\centering
\begin{tikzpicture}[
    vertex/.style={circle, draw, fill=white, inner sep=2pt, minimum size=6pt},
    edge/.style={thick},
    connector/.style={thick},
    scale=0.8
]
    \foreach \i in {0,...,4} { \DrawBlock{\i}{b_} }
    \ConnectBlocks{b_}

    \node[vertex, label={above:$z_0$}] (bz0) at (6.0, 2.5) {};

    \draw[edge] (b_t0) -- (bz0);
    \draw[edge] (b_t2) -- (bz0);
    \draw[edge] (b_t4) -- (bz0);

    \draw[edge] (b_t1) to[out=25, in=155] (b_t3);
\end{tikzpicture}
\caption{Example of an $LP_1$ snark with five building blocks.}
\label{fig:loupekine_lp1}
\end{subfigure}

\caption{Loupekine snarks composed of five building blocks.}
\label{fig:loupekine_types}
\end{figure}
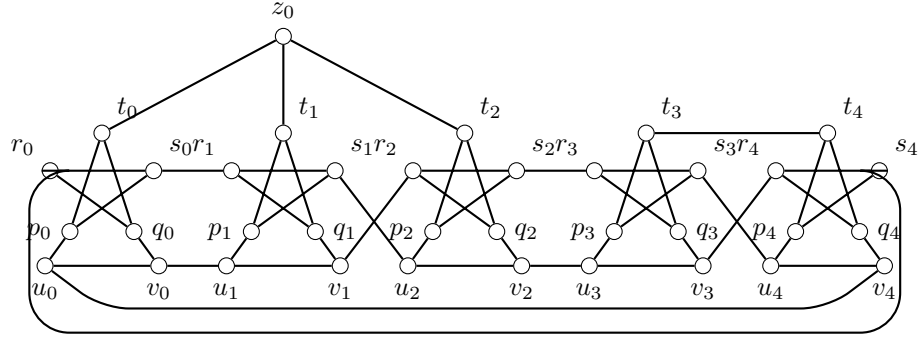
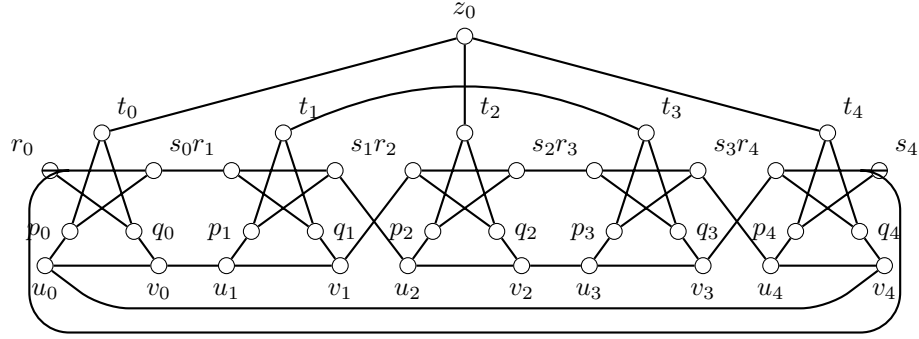

Next, we investigate the vertex cover number of parallel-crossed $LP_0$ and $LP_1$ Loupekine snarks. Our approach begins by analyzing a building block $B_i$ to determine the minimum vertex cover. Based on the local analysis of this subgraph, we establish the covering pattern for the entire graph composed of $n$ building blocks.

\begin{theorem}
\label{prop:loupekine_block}
Let $B_i$ be a building block of a Loupekine snark. Then, $\tau(B_i) = 4$.
\end{theorem}

\begin{proof}
    We first prove the lower bound $\tau(B_i) \geq 4$. Note that $B_i$ contains an induced 5-cycle $C_5 = (t_i, p_i, s_i, r_i, q_i)$ and an edge $u_i v_i$ that is not incident to $C_5$. By Lemma~\ref{lemma:cyclesVC}, the vertex cover number of an $n$-cycle is $\lceil n/2 \rceil$. Thus, covering the edges of $C_5$ requires at least $\lceil 5/2 \rceil = 3$ vertices. Since the edge $u_i v_i$ shares no vertices with $C_5$, covering it requires selecting at least one of its endpoints. Therefore, $\tau(B_i) \geq 3 + 1 = 4$.
    
    To establish the upper bound $\tau(B_i) \leq 4$, we can explicitly define valid vertex covers of cardinality 4 for $B_i$, such as $S_1 = \{t_i, s_i, q_i, u_i\}$ and $S_2 = \{t_i, r_i, p_i, v_i\}$, as illustrated in Figure~\ref{fig:cover_s1}, showing the cover $S_1$, and Figure~\ref{fig:cover_s2}, showing the cover $S_2$. Thus, $\tau(B_i) \leq |S_1| = 4$. 
    
    Since the lower and upper bounds match, we conclude that $\tau(B_i) = 4$.
\end{proof}

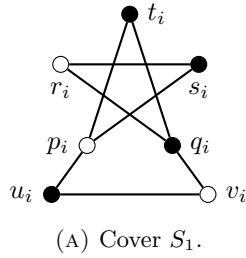
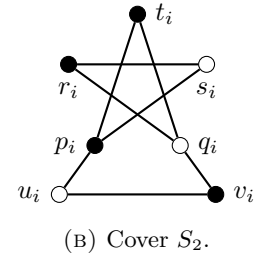
\begin{figure}
\def\Rin{1.2}
\def\Rout{2.2}

\begin{subfigure}[b]{0.48\textwidth}
\centering
\begin{tikzpicture}[
    vertex/.style={circle, draw, fill=white, inner sep=2pt, minimum size=6pt},
    covered/.style={circle, draw, fill=black, inner sep=2pt, minimum size=6pt},
    edge/.style={thick},
    scale=0.8
]
    \node[covered, label={right:$t_i$}] (l_t) at (90:\Rin) {};
    \node[vertex,  label={below:$r_i$}] (l_r) at (162:\Rin) {};
    \node[vertex,  label={left:$p_i$}]  (l_p) at (234:\Rin) {};
    \node[covered, label={right:$q_i$}] (l_q) at (306:\Rin) {};
    \node[covered, label={below:$s_i$}] (l_s) at (18:\Rin) {};

    \node[covered, label={left:$u_i$}]  (l_u) at (234:\Rout) {};
    \node[vertex,  label={right:$v_i$}] (l_v) at (306:\Rout) {};

    \draw[edge] (l_t) -- (l_p); \draw[edge] (l_p) -- (l_s);
    \draw[edge] (l_s) -- (l_r); \draw[edge] (l_r) -- (l_q);
    \draw[edge] (l_q) -- (l_t);
    \draw[edge] (l_p) -- (l_u); \draw[edge] (l_q) -- (l_v);
    \draw[edge] (l_u) -- (l_v);
\end{tikzpicture}
\caption{Cover $S_1$.}
\label{fig:cover_s1}
\end{subfigure}
\hfill
\begin{subfigure}[b]{0.48\textwidth}
\centering
\begin{tikzpicture}[
    vertex/.style={circle, draw, fill=white, inner sep=2pt, minimum size=6pt},
    covered/.style={circle, draw, fill=black, inner sep=2pt, minimum size=6pt},
    edge/.style={thick},
    scale=0.8
]
    \node[covered, label={right:$t_i$}] (r_t) at (90:\Rin) {};
    \node[covered, label={below:$r_i$}] (r_r) at (162:\Rin) {};
    \node[covered, label={left:$p_i$}]  (r_p) at (234:\Rin) {};
    \node[vertex,  label={right:$q_i$}] (r_q) at (306:\Rin) {};
    \node[vertex,  label={below:$s_i$}] (r_s) at (18:\Rin) {};

    \node[vertex,  label={left:$u_i$}]  (r_u) at (234:\Rout) {};
    \node[covered, label={right:$v_i$}] (r_v) at (306:\Rout) {};

    \draw[edge] (r_t) -- (r_p); \draw[edge] (r_p) -- (r_s);
    \draw[edge] (r_s) -- (r_r); \draw[edge] (r_r) -- (r_q);
    \draw[edge] (r_q) -- (r_t);
    \draw[edge] (r_p) -- (r_u); \draw[edge] (r_q) -- (r_v);
    \draw[edge] (r_u) -- (r_v);
\end{tikzpicture}
\caption{Cover $S_2$.}
\label{fig:cover_s2}
\end{subfigure}

\caption{Two examples of minimum vertex covers for the building block $B_i$.}
\label{fig:blocos_independentes}
\end{figure}

\begin{theorem}
\label{the:loupekine-snark-proof}
If $L_n$ is a parallel-crossed $LP_0$ or $LP_1$ Loupekine snark, then $\tau(L_n) = 4n$.
\end{theorem}

\begin{proof}
    To prove the upper bound $\tau(L_n) \leq 4n$, we demonstrate the existence of a vertex cover $S$ with cardinality $4n$. We define $S$ as the union of $n$ disjoint subsets $S_i$:
    
    $$S = \bigcup_{i=0}^{n-1} S_i, \quad \text{where } 
    S_i = \begin{cases} 
    \{t_i, s_i, q_i, u_i\} & \text{if } i \equiv 0,1 \pmod{4}; \\
    \{t_i, r_i, p_i, v_i\} & \text{if } i \equiv 2,3 \pmod{4}.
    \end{cases}$$
    
    Since the vertex sets of the building blocks $B_i$ are disjoint and each $S_i$ contains exactly four vertices, the total cardinality is $|S| = \sum_{i=0}^{n-1} |S_i| = 4n$. We now verify that $S$ covers all edges of $L_n$.

    Regarding the edges in the set $E(B_i)$, for $0\leq i \leq n-1$, Theorem~\ref{prop:loupekine_block} and its proof ensure that each $S_i$ is a vertex cover of the  block $B_i$. Specifically, for blocks where $i \equiv 0, 1 \pmod{4}$, the subset $\{t_i, s_i, q_i, u_i\}$ covers all the edges of $B_i$, as does the subset $\{t_i, r_i, p_i, v_i\}$ for blocks where $i \equiv 2, 3 \pmod{4}$.

    The cover also accounts for all link-edges between consecutive blocks under the parallel-crossed structure. In the parallel connection between $B_0$ and $B_1$, the edges $s_0r_1$ and $v_0u_1$ are covered by $s_0 \in S_0$ and $u_1 \in S_1$, respectively. In the subsequent crossed connection from $B_1$ to $B_2$, the edges $s_1u_2$ and $v_1r_2$ are covered by $s_1 \in S_1$ and $r_2 \in S_2$, respectively. For the direct connection from $B_2$ to $B_3$, the edge $v_2u_3$ is covered by $v_2 \in S_2$, while $s_2r_3$ is covered by $r_3 \in S_3$. Finally, in the crossed connection from $B_3$ to $B_0$, the edges $s_3u_0$ and $v_3r_0$ are covered by $u_0 \in S_0$ and $v_3 \in S_3$. This verification pattern repeats periodically.

    All upper link-edges are also covered because they are incident to vertices of type $t_i$, and $t_i \in S_i$ for all $i = 0, \dots, n-1$. Consequently, every edge in $L_n$ has at least one endpoint in $S$, which confirms that $\tau(L_n) \leq 4n$. The application of this cover to parallel-crossed $LP_0$ and $LP_1$ Loupekine snarks is illustrated in Figure~\ref{fig:loupekine_cover_lp0}, detailing the vertex cover of a parallel-crossed $LP_0$ snark, and Figure~\ref{fig:loupekine_cover_lp1}, detailing the vertex cover of a parallel-crossed $LP_1$ snark.

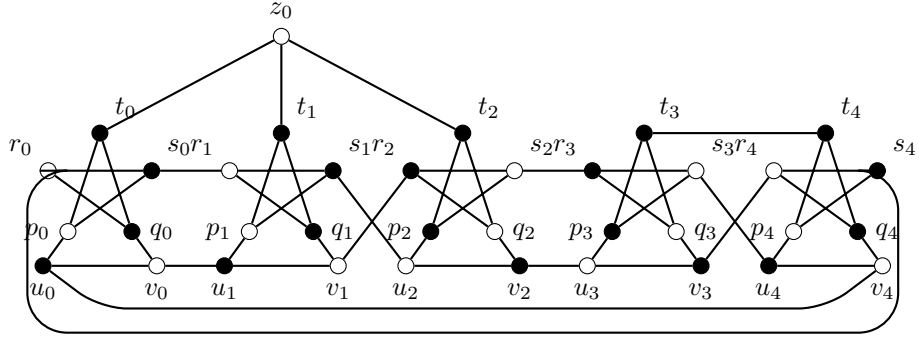
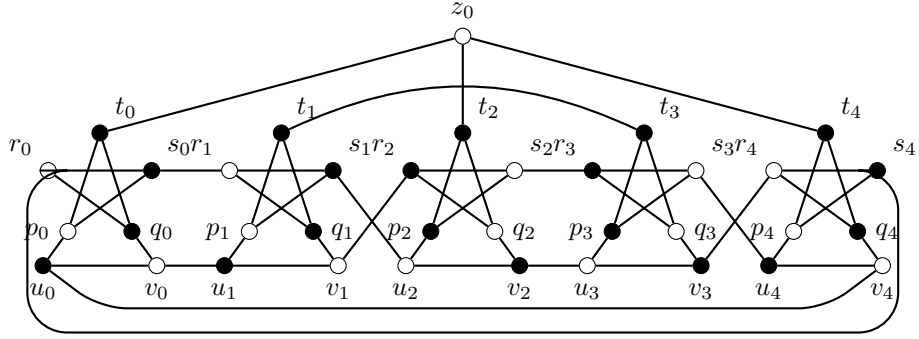
\begin{figure}
\centering

\def\Rin{0.9}
\def\Rout{1.6}
\def\BlockDist{3.0}
\def\LineDepthInner{-2.0}
\def\LineDepthOuter{-2.4}
\def\OuterMargin{1.2}

\newcommand{\DrawBlock}[2]{
    \begin{scope}[xshift=#1*\BlockDist cm]
        
        \def\colT{white} \def\colR{white} \def\colP{white} 
        \def\colQ{white} \def\colS{white} \def\colU{white} \def\colV{white}

        \ifnum #1=0 \def\colT{black} \def\colS{black} \def\colQ{black} \def\colU{black} \fi
        \ifnum #1=1 \def\colT{black} \def\colS{black} \def\colQ{black} \def\colU{black} \fi
        \ifnum #1=4 \def\colT{black} \def\colS{black} \def\colQ{black} \def\colU{black} \fi

        \ifnum #1=2 \def\colT{black} \def\colR{black} \def\colP{black} \def\colV{black} \fi
        \ifnum #1=3 \def\colT{black} \def\colR{black} \def\colP{black} \def\colV{black} \fi

        \node[vertex, fill=\colT, label={above right:$t_{#1}$}] (#2t#1) at (90:\Rin) {};
        \node[vertex, fill=\colR, label={above left:$r_{#1}$}] (#2r#1) at (162:\Rin) {};
        \node[vertex, fill=\colP, label={left:$p_{#1}$}] (#2p#1) at (234:\Rin) {};
        \node[vertex, fill=\colQ, label={right:$q_{#1}$}] (#2q#1) at (306:\Rin) {};
        \node[vertex, fill=\colS, label={above right:$s_{#1}$}] (#2s#1) at (18:\Rin) {};
        
        \node[vertex, fill=\colU, label={below:$u_{#1}$}] (#2u#1) at (234:\Rout) {};
        \node[vertex, fill=\colV, label={below:$v_{#1}$}] (#2v#1) at (306:\Rout) {};

        \draw[edge] (#2t#1) -- (#2p#1) -- (#2s#1) -- (#2r#1) -- (#2q#1) -- (#2t#1);
        \draw[edge] (#2p#1) -- (#2u#1);
        \draw[edge] (#2q#1) -- (#2v#1);
        \draw[edge] (#2u#1) -- (#2v#1);
    \end{scope}
}

\newcommand{\ConnectBlocks}[1]{
    \draw[edge] (#1s0) -- (#1r1); \draw[edge] (#1v0) -- (#1u1);
    \draw[edge] (#1s1) -- (#1u2); \draw[edge] (#1v1) -- (#1r2);
    \draw[edge] (#1s2) -- (#1r3); \draw[edge] (#1v2) -- (#1u3);
    \draw[edge] (#1s3) -- (#1u4); \draw[edge] (#1v3) -- (#1r4);

    \draw[edge, rounded corners=15pt] 
        (#1r0) -- (-\OuterMargin, {0.9*sin(162)})
               -- (-\OuterMargin, \LineDepthOuter)
               -- (4*\BlockDist + \OuterMargin, \LineDepthOuter)
               -- (4*\BlockDist + \OuterMargin, {0.9*sin(18)})
               -- (#1s4);

    \draw[edge, rounded corners=10pt] 
        (#1u0) -- (0, \LineDepthInner)
               -- (4*\BlockDist, \LineDepthInner)
               -- (#1v4);
}

\begin{subfigure}[b]{\textwidth}
\centering
\begin{tikzpicture}[
    vertex/.style={circle, draw, fill=white, inner sep=2pt, minimum size=6pt},
    covered/.style={circle, draw, fill=black, inner sep=2pt, minimum size=6pt},
    edge/.style={thick},
    connector/.style={thick},
    scale=0.8
]
    \foreach \i in {0,...,4} { \DrawBlock{\i}{a_} }
    \ConnectBlocks{a_}

    \node[vertex, label={above:$z_0$}] (az0) at (3.0, 2.5) {}; 
    \draw[edge] (a_t0) -- (az0);
    \draw[edge] (a_t1) -- (az0);
    \draw[edge] (a_t2) -- (az0);
    \draw[edge] (a_t3) -- (a_t4);
\end{tikzpicture}
\caption{Vertex cover of a parallel-crossed $LP_0$ snark.}
\label{fig:loupekine_cover_lp0}
\end{subfigure}

\vspace{0.8cm}

\begin{subfigure}[b]{\textwidth}
\centering
\begin{tikzpicture}[
    vertex/.style={circle, draw, fill=white, inner sep=2pt, minimum size=6pt},
    covered/.style={circle, draw, fill=black, inner sep=2pt, minimum size=6pt},
    edge/.style={thick},
    connector/.style={thick},
    scale=0.8
]
    \foreach \i in {0,...,4} { \DrawBlock{\i}{b_} }
    \ConnectBlocks{b_}

    \node[vertex, label={above:$z_0$}] (bz0) at (6.0, 2.5) {};

    \draw[edge] (b_t0) -- (bz0);
    \draw[edge] (b_t2) -- (bz0);
    \draw[edge] (b_t4) -- (bz0);

    \draw[edge] (b_t1) to[out=25, in=155] (b_t3);
\end{tikzpicture}
\caption{Vertex cover of a parallel-crossed $LP_1$ snark.}
\label{fig:loupekine_cover_lp1}
\end{subfigure}

\caption{Minimum vertex covers for parallel-crossed $LP_0$ and $LP_1$ Loupekine snarks.}
\label{fig:loupekine_cover}
\end{figure}

To establish the lower bound $\tau(L_n) \geq 4n$, observe that the vertex sets $V(B_0), V(B_1), \dots, V(B_{n-1})$ are pairwise disjoint. Any vertex cover of $L_n$ must cover the internal edges of each block $B_i$, which, by Theorem~\ref{prop:loupekine_block}, requires at least 4 vertices from $V(B_i)$. 
Thus, by Lemma~\ref{lemma:VCsubgraph}, we have that $\tau(L_n) \geq 4n$. Combining this with the upper bound, we conclude that $\tau(L_n) = 4n$.
\end{proof}

\section{Concluding Remarks}

While the vertex cover problem is a well-studied topic in graph theory and computational complexity, the study of the parameter $\tau(G)$ remains open for many graph classes, particularly those with a maximum degree of three or higher. This work contributes to the field by establishing the vertex cover number for several families of snarks: flower snarks, Goldberg snarks, the first family of generalized Blanu\v{s}a snarks, and parallel-crossed $LP_0$ and $LP_1$ Loupekine snarks. Our main results are summarized in Table~\ref{tab:snarks_dr}. We also proved that the problem of determining whether an arbitrary snark has a vertex cover of size at most $k$ is NP-complete. Although the Vertex Cover Problem remains NP-complete when restricted to this class, as future work, we propose studying this parameter for other subclasses of snarks.

\begin{table}
    \centering
    \caption{Vertex cover number $\tau(G)$ for the studied snark families.}
    \vspace{0.2cm}
    \renewcommand{\arraystretch}{1.5}
    \begin{tabular}{@{} l c @{}}
        \toprule
        \textbf{Snark Family} & \boldmath{$\tau(G)$} \\
        \midrule
        Flower snark $F_n$ & $2n+1$ \\
        Goldberg snark $G_n$ & $\left\lceil \frac{9n}{2} \right\rceil$ \\
        First family of generalized Blanu\v{s}a snarks $\mathfrak{B}^1_n$ & $\left\lceil \frac{14n + 17}{3} \right\rceil$ \\
        Parallel-crossed $LP_0$ and $LP_1$ Loupekine snarks $L_n$ & $4n$ \\
        \bottomrule
    \end{tabular}
    \label{tab:snarks_dr}
\end{table}

\section{Acknowledgments}

This work was partially supported by Conselho Nacional de Desenvolvimento Científico e Tecnológico – CNPq.

\bibliographystyle{plain} 
\bibliography{mybib}

\end{document}